\documentclass[leqno,11pt]{amsart}

\usepackage{amssymb}
\usepackage{amsmath}
\usepackage{enumerate}
\usepackage[pdftex]{graphicx}
\usepackage{graphicx, color}
\usepackage[inner=2.5cm,outer=2.5cm,top=2.5cm,bottom=2.5cm]{geometry}

\usepackage{multirow}

\def\s{\mathbb{S}}

\def\R{\mathbb{R}}

\def\L{\mathbb{L}}
\def\E{\mathbb{E}}
\def\H{\mathbb{H}}
\def\K{\mathcal{K}}

\newcommand{\arcsinh}{\mathop{\rm arcsinh }\nolimits}
 
\DeclareMathOperator{\arccosh}{arccosh}

\newtheorem{theorem}{Theorem}[section]
\newtheorem{proposition}[theorem]{Proposition}
\newtheorem{corollary}[theorem]{Corollary}

\theoremstyle{definition}
\newtheorem{definition}[theorem]{Definition}
\newtheorem{remark}[theorem]{Remark}
\newtheorem{example}[theorem]{Example}

\numberwithin{equation}{section}

\begin{document}

\title[Rotational Weingarten surfaces in Lorentz-Minkowski space]{Rotational Weingarten surfaces \\ in Lorentz-Minkowski space}

\author[P. Carretero]{Paula Carretero}
\address{Departamento de Matem\'aticas \\
	Universidad de Ja\'{e}n \\
	23071 Ja\'{e}n, Spain.}
\email{pch00005@red.ujaen.es}

\author[I. Castro]{Ildefonso Castro}
\address{Departamento de Matem\'{a}ticas \\
	Universidad de Ja\'{e}n \\
	23071 Ja\'{e}n, Spain and IMAG, Instituto de Matem\'aticas de la Universidad de Granada, 18071 Granada, Spain.}
\email{icastro@ujaen.es}

\author[I. Castro-Infantes]{Ildefonso Castro-Infantes}
\address{Departamento de Matem\'{a}ticas  \\
	Universidad de Murcia \\
	30100 Murcia, Spain. } 
\email{ildefonso.castro@um.es}




\date{}

\begin{abstract}
	We propose a new approach to the study of rotational surfaces in Lorentz-Minkowski space based on the notion of the geometric linear momentum of the generatrix curves with respect to the axes of revolution. This technique allows us to reduce any Weingarten condition on the surface to a first-order ordinary differential equation for the momentum as a function of the distance to the corresponding axis, providing a unified framework that encompasses the three causal types of rotation axes. As a direct application, we classify important families of rotational Weingarten surfaces in this setting, including some linear and quadratic cases. Furthermore, we introduce the non-degenerate quadric surfaces of revolution in Lorentz-Minkowski space and characterize them in terms of a specific cubic Weingarten relation.
\end{abstract}

\maketitle



\section{Introduction}\label{SectIntro}

The study of surfaces defined by properties of their curvature is a central theme in classical differential geometry. Among these, \textit{Weingarten surfaces} ---those satisfying a non-trivial functional relation $W(\kappa_1, \kappa_2)=0$ between their principal curvatures--- have played a distinguished role since the seminal works of Weingarten \cite{W61}, Chern \cite{Ch45} and Hopf \cite{H51}. While the classification of such surfaces in the general case remains an open problem, \textit{rotational surfaces} in 3-manifolds serve as a fundamental testing ground (cf.\ \cite{KS05}). Due to their symmetry, they reduce the complexity of the non-linear partial differential equations governing curvature problems to manageable ordinary differential equations, allowing for the discovery of explicit examples and complete classifications.

In the context of the Lorentz-Minkowski space $\mathbb{L}^3$, the geometry of rotational surfaces presents richer and more distinct phenomena than in the Euclidean setting, primarily due to the causal character of the axis of revolution (which may be spacelike, timelike, or lightlike) and the metric signature. We  recall (see \cite{HN83}) the relevant role that rotational surfaces in $\L^3$ have played in the global theory
of surfaces in $\L^3$. However, the analytical difficulties in obtaining classification results on Weingarten surfaces even exceed those encountered in Euclidean space  (see e.g.\ \cite{AG03}, \cite{BL11}, \cite{BLS11}, \cite{CL97}, \cite{dS21}, \cite{Du10}, \cite{HN84}, \cite{LV06}, \cite{LV07}, \cite{YZ21}). 

To overcome these challenges, we adopt a novel geometric approach recently developed for the Euclidean case in \cite{CC22} and \cite{CC24} and for the spherical case in \cite{CCIs24}. This method relies on the key analytical tool for our study: the \textit{geometric linear momentum} of the Lorentzian or Euclidean generatrix curve with respect to the axis of revolution (see Section \ref{SectMomentum}). 
Geometrically, it controls the angle of the Frenet frame of the curve with the coordinate axes and, 
in physical terms, it may be described as the linear momentum (with respect to the rotation axis) of a particle of unit mass with unit speed and trajectory the corresponding curve. In addition, it can be interpreted as an antiderivative (up to sign) of the curvature of the curve, when this is expressed as a function of the (pseudo)distance to a fixed line. The most noteworthy fact (see Theorem \ref{Th:Kdetermine} and Corollary \ref{cor:Param}) is that the  geometric linear momentum determines the curve ---and therefore the rotational surface--- explicitly by quadratures (modulo  translations along the axis of revolution).

As demonstrated in previous works in the Euclidean setting (see e.g.\ \cite{CC22}, \cite{CC23}, \cite{CC24}, \cite{CCC25}), rewriting the curvature conditions in terms of this momentum transforms the problem into a first-order differential equation, providing a unified and powerful framework for analysis. Therefore
the aim of this paper is to apply this momentum-based technique to classify rotational Weingarten surfaces in $\mathbb{L}^3$, that is, those satisfying $W(k_{\text m}, k_{\text p})=0$, where $k_{\text m}$ and $k_{\text p}$ stand for the principal curvatures along meridians and parallels, respectively. Along  these lines, our key result  (see both Theorem \ref{Th:kmkp} and Corollary \ref{cor:edoW}) reformulates any Weingarten relation on any type  of rotational surface as a first-order differential equation  for the geometric linear momentum   of its generatrix with respect to the axis of revolution as a function of  the (pseudo)distance  from the curve to the axis.

As a first direct application, we provide (see Theorem \ref{Th:H=0}) a considerably shorter proof of the classification of the rotational surfaces with zero mean curvature in $\L^3$, obtained by Kobayashi \cite{Ko83} in the spacelike case and by Van de Woestijne \cite{VW90} in the timelike case. 
Furthermore, drawing inspiration from \cite{CC24}, we introduce (Definition \ref{def:Hopf}) a family of rotational surfaces in $\mathbb{L}^3$ that acts as the Lorentzian analogue of the classical Hopf surfaces in $\mathbb{E}^3$ \cite{H51}. This leads to Theorem \ref{Th:Hopf LW}, which constitutes a significant generalization of the preceding result; specifically, we characterize Lorentzian Hopf surfaces as the only rotational surfaces (other than planes) in $\mathbb{L}^3$ fulfilling the linear Weingarten relation $k_{\text m} = q\, k_{\text p}$,  $q \neq 0$ (Theorem \ref{Th:H=0} being the specific instance where $q=-1$). This approach offers a description distinct from that found in Theorem 1.3 of \cite{BL11} for spacelike surfaces and completes the classification for the timelike case.

We also address certain quadratic Weingarten condition (see Theorem \ref{Th:Wquadratic}), providing a complete description of the family satisfying  $k_{\rm m}=\mu \, k_{\rm p}^2$, $\mu \neq 0$, highlighting the interplay between the causal character of the rotation axis and the algebraic structure of the Weingarten relation.

 Finally, to the best of our knowledge, there is a scarcity of literature concerning non-degenerate quadric surfaces of revolution in $\mathbb L^3$. In Section \ref{SectQuadric}, we identify the Lorentzian analogues of the non-degenerate quadric surfaces of revolution in $\mathbb E^3$. Subsequently, we investigate the cubic Weingarten relation $k_{\text m}=\mu\, k_{\text p}^3$, $\mu \neq 0$, and prove in Theorem \ref{Th:W-cubic} that this condition characterizes the non-degenerate quadric surfaces of revolution in $\mathbb{L}^3$. This result provides a Lorentzian counterpart to the findings recently obtained in Euclidean space in \cite{CC22} and \cite{CC23}.

\section{Preliminaries}\label{SectPreliminaries}

The Lorentz-Minkowski space $\L^3$ is defined as the vector space $\R^3$ endowed
with the Lorentzian metric $\langle \cdot , \cdot \rangle \!=\!dx_1^2+dx_2^2-dx_3^2$, where $(x_1, x_2, x_3)$
are canonical coordinates in $\R^3$. 
For any $u,v\in \L^3$, the Lorentzian cross-product $u \times v$ is defined in such a way that the relation $\langle u \times v, w \rangle = \det (u,v,w)$ holds for all $w\in \L^3$. We may introduce the concept of causal character as follows: A vector $w\in \L^3$ is said to be \textit{spacelike} if $\langle w, w \rangle >0$ or $w=0$, \textit{timelike} if $\langle w, w \rangle <0$ and \textit{lightlike} (or \textit{null}) if $\langle w, w \rangle =0$ and $w\neq 0$. The \textit{light-cone} $\mathcal C$ is the set of all lightlike vectors of $\L^3$.
Given a vector subspace $U\subset \R^3$, $U $ is called \textit{spacelike} if the induced metric by $\langle \cdot , \cdot \rangle $ on $U$ is positive definite, \textit{timelike} if it has index one, and \textit{lightlike} if it is degenerate.

\subsection{Surfaces in $\L^3$}\label{SectSurfaces}

A non-degenerate surface $\Sigma $ in $\L^3$ is called spacelike (resp.\ timelike) if the induced metric, also denoted by $\langle \cdot , \cdot \rangle$, is positive definite (resp.\  is a metric with index 1). Equivalently, $\Sigma $ is spacelike (resp.\ timelike) if all tangent planes to $\Sigma $ endowed with the induced metric are spacelike (resp.\ timelike). 

The \textit{de Sitter 2-space} of radius $R>0$ is the timelike surface given by $\s^2_1(R)=\{  x \in \L^3 \, | \, \langle x,x \rangle =R^2 \}$ and the\textit{ hyperbolic plane} of radius $R>0$ is the spacelike  surface  given by $\H^2_+(R)=\{  x \in \L^3 \, | \, \langle x,x \rangle =-R^2, \ x_3 >0 \}$. They are both totally umbilical surfaces of $\L^3$ with constant Gauss curvature $1/R^2$ and $-1/R^2$, respectively. See Figure \ref{fig:LorentzSpheres}.

\begin{figure}[h]
	\begin{center}
		\includegraphics[height=5cm]{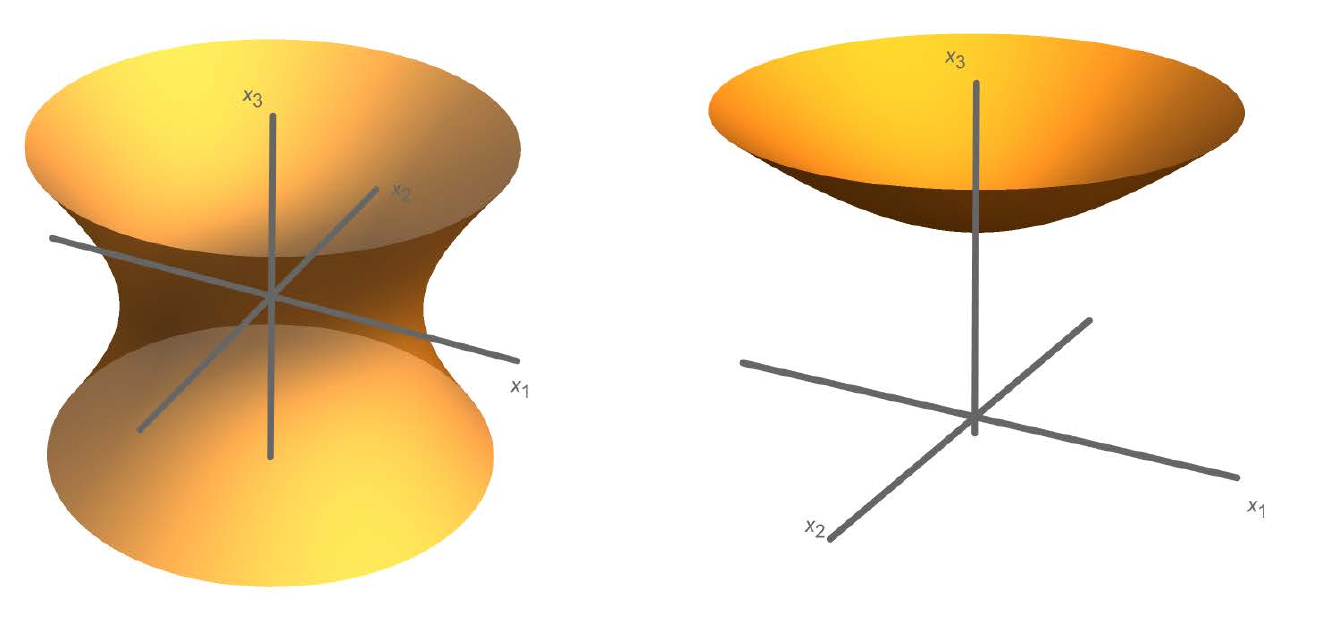}
		\caption{De Sitter 2-space and hyperbolic plane in $\L^3$.}
		\label{fig:LorentzSpheres}
	\end{center}
\end{figure}

Following \cite{Lo14}, we recall the local study of a non-degenerate surface $\Sigma$ in $\L^3$ with local unit normal vector field $\nu$ such that $\langle \nu, \nu \rangle =-\epsilon$, where $\epsilon =1$ if $\Sigma $ is spacelike and $\epsilon =-1$ if $\Sigma $ is timelike. We use a local parametrization $\Sigma \equiv X=X(u,v)$, $(u,v)\in U \subseteq \R^2$, of the (spacelike or timelike) immersion $\Sigma \rightarrow \L^3 $. With respect to the local basis $\{ X_u, X_v\}$, the coefficients of the first 
fundamental form are given by 
$$E=\langle X_u,X_u \rangle, \ F=\langle X_u,X_v \rangle , \ G=\langle X_v,X_v \rangle .$$
Denote $W:=EG-F^2$. Then  $\Sigma $ is spacelike if $W>0 $, and timelike if $W<0$. The
unit normal vector field $$\nu=\frac{X_u \times X_v}{|X_u \times X_v|}=\frac{X_u \times X_v}{\sqrt{\epsilon W}},$$ verifies $\langle \nu, \nu \rangle =-\epsilon$. We remark that the sign of $\epsilon $ here is the opposite of the one in \cite{Lo14}.
The coefficients of the second fundamental form with respect to $\{ X_u, X_v\}$ are given by $$e=\langle X_{uu},\nu \rangle, \ f=\langle X_{uv},\nu \rangle , \ g=\langle X_{vv},\nu \rangle .$$
Moreover, the mean curvature $H$ and the Gauss curvature $  K_{\rm G}$ of $\Sigma $ are expressed by
 $$H=-\frac{\epsilon}{2}\,\frac{eG-2fF+gE}{EG-F^2}, \quad  K_{\rm G}=-\epsilon \, \frac{eg-f^2}{EG-F^2} .$$
If $H^2+\epsilon K_{\rm G} > 0$, the Weingarten endomorphism $A_\nu$ is diagonalizable and then
\begin{equation}\label{eq:H}
H=-\frac{\epsilon}{2}(\kappa_1+\kappa_2), \quad K_{\rm G}=-\epsilon \, \kappa_1\kappa_2, 
\end{equation}
where $\kappa_1, \kappa_2$ are the principal curvatures. When $A_\nu$ is diagonalizable,  
the condition $H^2+\epsilon K_{\rm G} = 0$ characterizes the umbilical points of the surface.

\subsection{Planar curves in $\L^3$}\label{SectCurves}
If we consider planar curves immersed in $\L^3$, there are three possibilities depending if the corresponding plane is spacelike, timelike of lightlike. For our purposes in this article, it will be enough to consider only the following first two cases. 

Let $\gamma =(x,y):I\subseteq \R \rightarrow \E^2$ be a curve lying in the Euclidean $xy$-plane. We consider that $\gamma $ is arc-length parameterized.  Denote by $\, \dot {} \,$ the derivative with respect to the arc parameter. Then the Frenet frame of $\gamma $ is given by
$T=\dot \gamma =(\dot x, \dot y)$, $N=J \dot \gamma  =(-\dot y, \dot x)$, where $T$ (resp.\ $N$) is the unit tangent (resp.\ normal) vector of $\gamma$.  
Since $ \dot x^2 + \dot y^2=1$, we can write $T=(\cos \theta, \sin \theta)$, for a certain angle function $\theta$. The Frenet equations of $\gamma$ are written as $\dot T=\kappa \, N$, $\dot N= -\kappa \, T$, where $\kappa $ is the (signed) curvature given by
$\kappa =  \langle \dot T, N \rangle =  \det (\dot \gamma , \ddot \gamma)= \dot x \ddot y-\ddot x \dot y= \dot \theta$.


In a similar way, let $\gamma =(x,z):I\subseteq \R \rightarrow \L^2$ be a unit speed curve lying in the Lorentzian $xz$-plane. The Frenet frame of $\gamma $ is now given by
$T=\dot \gamma =(\dot x, \dot z)$, $N=\dot \gamma ^ \perp =(\dot z, \dot x)$. The unit speed condition here means that $\langle T, T \rangle =\epsilon $, where $\epsilon =1$ if $\gamma $, i.e.\ $T$, is spacelike, and $\epsilon =-1$ if $\gamma $, i.e.\ $T$, is timelike.
We point out that $ \langle N, N \rangle =-\epsilon$. Since $ \dot x^2-\dot z^2= \epsilon$, $\epsilon =\pm 1$, we can write  $T=(\cosh \varphi, \sinh \varphi)$ when $\epsilon =1$, and  $T=(\sinh \phi, \cosh \phi)$ when $\epsilon =-1$, for certain angle functions $\varphi $ and $\phi$.
The Frenet equations of $\gamma$ are written as $\dot T=\kappa \, N$, $\dot N= \kappa \, T$, where $\kappa $ is the (signed) curvature given by
$\kappa = -\epsilon \langle \dot T, N \rangle = \epsilon \det (\dot \gamma , \ddot \gamma)= \epsilon (\dot x \ddot z-\ddot x \dot z)$; in addition,  $\kappa =\dot \varphi$ if $\epsilon =1$, and $\kappa =\dot \phi$ if $\epsilon =-1$.
If we introduce new coordinates $(u,v)$ in $\L^2$ by means of
$u=x+z$, $v=x-z$, and write $\gamma =(u,v)$, then we have that $\dot u \dot v =\epsilon$. Since now $\dot \gamma =(\dot u,\dot v)$, 
we get $\dot \gamma =(e^\varphi,e^{-\varphi})$ if $\epsilon =1$ , and we arrive at $\dot \gamma =(e^\phi,-e^{-\phi})$ if $\epsilon =-1$.


\section{The geometric linear momentum of a planar curve (with respect to an axis)}\label{SectMomentum}

We now introduce the fundamental tool for the innovative study of the rotational surfaces of $\L^3$ that we will develop in the next section:
the geometric linear momentum of a planar curve (the generatrix curve) with respect to a line (the axis of revolution). In the interest of making this article self-contained, we compile below in a schematic way the main results concerning this topic which can be found in more detail in \cite{CCI16}, \cite{CCI18}, \cite{CCI20} and \cite{CCIs20}, which will be essential in the next section. We remark that this tool was crucial in the aforementioned references for the study of plane curves whose curvature depends on the distance to a line (a particular case of a more general problem posed by D.~Singer in \cite{S99}).

\subsection{Euclidean curve} Let $\gamma =(x,y):I\subseteq \R \rightarrow \E^2$ be a unit speed curve, i.e.\ $\dot x^2+\dot y^2= 1$, such that
$\kappa=\kappa(y)$. We remark that $y$ is the distance from $\gamma $ to the $x$-axis. We follow the notation of \S \ref{SectCurves} and define the\textit{ geometric linear momentum of $\gamma$ with respect to the $x$-axis} by 
$$ \mathcal K(s)= \dot x (s)  = \cos \theta (s) \in [-1,1],$$
where $\, \dot {} \,$ denotes the derivative with respect to the arc parameter $s$.
We point out that $\mathcal K$  is well defined, up to the sign, depending on the orientation of $\gamma$.
Using the Frenet equations of $\gamma $, we easily conclude that 
$$ d \mathcal K = -\kappa(y) dy.$$
When $y$ is non constant, we can determine the curve $\gamma$
by integrating the separable o.d.e.\ $ \mathcal K^2+\dot y^2= 1$, obtaining
$$  s = s(y) =  \int \!\frac{d y}{\sqrt{1-\mathcal K(y)^2}},$$ and invert it to get $ y=y(s) $. In addition, from the definition of $ \mathcal K$, we have that
$$ x=x(s)=   \int \! \mathcal K (y(s))d s. $$
Eliminating  $s$ in the above two integrals, we determine  $\gamma$ (up to translations along the $x$-axis) as the graph  
$$x=x(y)= \int \frac{\mathcal K(y)\,dy}{\sqrt{1-\mathcal K(y)^2}}= \int \cot \theta (y)\,dy.$$
As a summary, \textit{any Euclidean curve $\gamma =(x,y)$, with $y$ non-constant, is uniquely determined by its geometric linear momentum $\mathcal K$ as a function of the distance from $\gamma $ to the $x$-axis, that is, by $\mathcal K= \mathcal K(y)$. 
The uniqueness is modulo translations in the $x$-direction.}

\subsection{Lorentzian curve} Let $\gamma =(x,z):I\subseteq \R \rightarrow \L^2$ be a unit speed curve. This means that  $\dot x^2-\dot z^2= \epsilon$, where $\epsilon =1$ if $\gamma $ is spacelike, and $\epsilon =-1$ if $\gamma $ is timelike. We use again the notation of \S \ref{SectCurves}.

\subsubsection{Case $\kappa=\kappa(x)$} We point out that $|x|$ is the maximum Lorentzian pseudodistance through spacelike geodesics from $\gamma$ to the $z$-axis.
We define the\textit{ geometric linear momentum of $\gamma$ with respect to the $z$-axis} by 
$$ \mathcal K (s)= \dot z (s) = \left\{  
	\begin{array}{ll}
	\sinh \varphi (s) \in (-\infty,+\infty) & {\rm if \ } \epsilon =1 \\	
	\cosh \phi  (s) \geq 1 & {\rm if \ } \epsilon =-1
	\end{array}
		\right. ,$$
where $\, \dot {} \,$ denotes the derivative with respect to the arc parameter $s$. We point out that $\mathcal K$  is well defined, up to the sign, depending on the orientation of $\gamma$.
Using the Frenet equations of $\gamma $, we easily conclude that 
$$ d \mathcal K = \kappa(x) dx.$$
When $x$ is non constant, we can determine the curve $\gamma$
by integrating the separable o.d.e.\ $ \dot x^2 - \mathcal K^2= \epsilon$, obtaining
	$$  s = s(x) =  \int \!\frac{d x}{\sqrt{\mathcal K(x)^2+\epsilon}},$$
	and invert it to get $ x=x(s) $. In addition, from the definition of $ \mathcal K$, we have that
	$$ z=z(s)=   \int \! \mathcal K (x(s))d s .$$
	Eliminating  $s$ in the above two integrals, we determine  $\gamma$ (up to translations along the $z$-axis) as the graph 
	$$z=z(x)= \int \frac{\mathcal K(x)\,dx}{\sqrt{\mathcal K(x)^2+\epsilon}} =
	\left\{  
	\begin{array}{ll}
		\int \tanh \varphi (x) \, dx & {\rm if \ } \epsilon =1 \\	\\
		\int \coth \phi (x) \, dx & {\rm if \ } \epsilon =-1
	\end{array}
	\right.
	$$
As a summary, \textit{any spacelike or timelike curve $\gamma =(x,z)$, with $x$ non-constant, is uniquely determined by its geometric linear momentum $\mathcal K$ as a function of the pseudodistance from $\gamma $ to the $z$-axis, that is, by $\mathcal K= \mathcal K(x)$. 
	The uniqueness is modulo translations in the $z$-direction.}
	
\subsubsection{Case $\kappa=\kappa(z)$} We point out that $|z|$ is the maximum Lorentzian pseudodistance through timelike geodesics from $\gamma$ to the $x$-axis.
We define the\textit{ geometric linear momentum of $\gamma$ with respect to the $x$-axis} by	
$$ \mathcal K(s)= \dot x (s) =\left\{  
	\begin{array}{ll}
		\cosh \varphi (s) \geq 1 & {\rm if \ } \epsilon =1 \\	
		\sinh \phi (s) \in (-\infty,+\infty)  & {\rm if \ } \epsilon =-1
	\end{array}
	\right. ,$$
	where $\, \dot {} \,$ denotes the derivative with respect to the arc parameter $s$. We point out that $\mathcal K$  is well defined, up to the sign, depending on the orientation of $\gamma$.
	Using the Frenet equations of $\gamma $, we easily conclude that 
	$$ d \mathcal K = \kappa(z) dz.$$
	When $z$ is non constant, we can determine the curve $\gamma$
	by integrating the separable o.d.e.\ $  \mathcal K^2- \dot z^2 = \epsilon$, obtaining
	$$  s = s(z) =  \int \!\frac{d z}{\sqrt{\mathcal K(z)^2-\epsilon}},$$ and invert it to get $ z=z(s) $.
	In addition, from the definition of $ \mathcal K$, we have that
	$$ x=x(s)=   \int \! \mathcal K (z(s))d s .$$
Eliminating  $s$ in the above two integrals, we determine  $\gamma$ (up to translations along the  $x$-axis) as the graph 
		$$x=x(z)= \int \frac{\mathcal K(z)\,dz}{\sqrt{\mathcal K(z)^2-\epsilon}}=
		\left\{  
	\begin{array}{ll}
		\int \coth \varphi (z) \, dz & {\rm if \ } \epsilon =1 \\	\\
		\int \tanh \phi (z) \, dz & {\rm if \ } \epsilon =-1
	\end{array}
	\right.$$
	As a summary, \textit{any spacelike or timelike curve $\gamma =(x,z)$, with $z$ non-constant, is uniquely determined by its geometric linear momentum $\mathcal K$ as a function of the pseudodistance from $\gamma$ to the $x$-axis, that is, by $\mathcal K= \mathcal K(z)$. 
		The uniqueness is modulo translations in the $x$-direction.}
	
\subsubsection{Case $\kappa=\kappa(v)$, with $v=x-z$ and $u=x+z$} We point out that $|v|$ is the Lorentzian pseudodistance from $\gamma$ to the lightlike geodesic $x=z$ through the horizontal spacelike geodesic or the vertical timelike geodesic. 
We define the \textit{geometric linear momentum of $\gamma=(u,v)$ with respect to the $u$-axis} by	
$$ \mathcal K(s)= \dot v (s) =\left\{  
	\begin{array}{ll}
		\  \  e^{-\varphi (s)} > 0& {\rm if \ } \epsilon =1 \\	
		-e^{-\phi (s) } < 0 & {\rm if \ } \epsilon =-1
	\end{array}
	\right. ,$$
		where $\, \dot {} \,$ denotes the derivative with respect to the arc parameter $s$. We point out that $\mathcal K$  is well defined, up to the sign, depending on the orientation of $\gamma$.
	Using the Frenet equations of $\gamma $, we conclude that 
	$$ d \mathcal K = -\kappa(v) dv.$$
	The definition of $ \mathcal K$ implies that
	$$  s = s(v) =  \int \!\frac{d v}{\mathcal K(v)}$$ and, inverting it, we get $ v=v(s) $.
			We can finally determine the curve $\gamma =(u,v)$
	by integrating the separable o.d.e.\ $  \dot u \, \mathcal K= \epsilon$, obtaining 
	$$ u=u(s)=  \epsilon \int \! \frac{ds}{\mathcal K (v(s))} . $$
	Eliminating  $s$ in the above two integrals, we determine  $\gamma$ (up to translations along the $u$-axis) as the graph
		$$u=u(v)=\epsilon \int \! \frac{dv}{\mathcal K (v)^2} =
		\left\{  
		\begin{array}{ll}
			\int e^{2\varphi(v)} \, dv & {\rm if \ } \epsilon =1 \\	\\
		-	\int e^{2\phi(v)} \, dv & {\rm if \ } \epsilon =-1
		\end{array}
		\right.$$
		As a summary, \textit{any spacelike or timelike curve $\gamma =(u,v)$ is uniquely determined by its geometric linear momentum $\mathcal K$ as a function of the pseudodistance from $\gamma $ to the $u$-axis, that is, by $\mathcal K= \mathcal K(v)$. 
			The uniqueness is modulo translations in the $u$-direction.}

\begin{example}\label{ex:geodesics}
Suppose $\kappa =0$. Then the geometric linear momentum $\mathcal K$ must be constant, i.e.\ $\theta\equiv \theta_0 \in (0,\pi)$, $\varphi\equiv\varphi_0 \in \R$ and $\phi\equiv \phi_0 \in \R$. The above computations directly arrive in $\E^2$ at the lines $x = \cot \theta_0 \, y $, and the spacelike geodesics $z=\tanh \varphi_0 \, x$ (when $\epsilon =1$) and the timelike geodesics $z=\coth \phi_0 \, x$ (when $\epsilon =-1$)  in $\L^2$.
\end{example}


\section{A new approach to rotational surfaces in $\L^3$}\label{SectRotational}

A surface in $\L^3 $ is called a rotational surface (or a surface of revolution) with axis $l$ if it is invariant under the action of rotations with axis $l$, i.e.\  isometries in $\L^3$ which fix point-wise the straight line $l$. Rotations in $\L^3$ are completely determined by the causal character of the rotation axis.
Thus we study non-degenerate rotational surfaces in $\L^3 $ with spacelike, timelike or lightlike \textit{axis of revolution} $l$ and \textit{generatrix curve} $\gamma $ lying in a non-degenerate plane containing $l$ in such a way that $l$ does not meet the curve $\gamma $. In this way, the rotational surface $ S_\gamma^l$ is defined as the orbit of $\gamma $ under the orthogonal transformation of $\L^3 $ with positive determinant that leaves $l$ fixed (see e.g.\ \cite{dS21}, \cite{Du10} or \cite{Lo14}). We will use the notions of \textit{meridians} and \textit{parallels} in the same way that for a surface of revolution in $\E^3$. Since we only consider non-degenerate rotational surfaces, it is enough to consider the case that $\gamma $ is spacelike or timelike. Next we will give explicit parametrizations in all possible cases.
There are three types of 1-parameter
subgroups of isometries of $\L^3$ which leave a line point-wise fixed. We call them
hyperbolic, elliptic or parabolic rotations depending on whether the line fixed (the axis
of revolution) is spacelike, timelike or lightlike. Hence we distinguish the following classes of rotational surfaces in $\L^3$:

\subsection{Hyperbolic rotational surfaces} We can take the spacelike axis as the $x_1$-axis  $l={\rm span}(1,0,0)$. 
In this case, we must also consider two sub-cases according to the generatrix curve is Lorentzian or Euclidean:
\subsubsection{Hyperbolic rotational surfaces of first type:}\label{SectSxz_x} The generatrix curve is Lorentzian. Let $\gamma (s)=(x(s),z(s))$, $z(s)>0$, $s\in I \subseteq \R$. Then $ S_\gamma^l$ will be denoted by  $  { S_{(x,z)}^{x_1}}$ and is given by
$$ X(s,t)=\left( x(s), z(s) \sinh t, z(s) \cosh t \right), \ s\in I, \ t \in \R  . $$
\subsubsection{Hyperbolic rotational surfaces of second type:}\label{SectSxy_x} The generatrix curve is Euclidean. Let $\gamma (s) =(x(s),y(s))$, $y(s)>0$, $s\in I \subseteq \R$. Then $ S_\gamma^l$ will be denoted by $  {S_{(x,y)}^{x_1}}$ and is given by
$$ X(s,t)=\left( x(s), y(s) \cosh t, y(s) \sinh t \right), \ s\in I, \ t \in \R  . $$

\subsection{Elliptic rotational surfaces}\label{SectSxz_z} We can take the timelike axis as the $x_3$-axis $l={\rm span}(0,0,1)$.  Let $\gamma (s)=(x(s),z(s))$, $x(s)>0$, $s\in I \subseteq \R$. Then $ S_\gamma^l$ will be denoted by $  {S_{(x,z)}^{x_3}}$  and is given by
$$ X(s,t)=\left( x(s) \cos t, x(s) \sin t, z(s)  \right), \ s\in I, \ t \in (0,2\pi)  . $$

\subsection{Parabolic rotational surfaces}\label{SectSuv_u} We can take the lightlike as the $u$-axis $l={\rm span}(1,0,1)$, recalling that $u=x_1+x_3$, $v=x_1-x_3$. Let $\gamma (s)=(x(s),z(s))=(u(s),v(s))$, $v(s)>0$, $s\in I \subseteq \R$. Then $ S_\gamma^l$ will be denoted by $  {S_{(u,v)}^u}$  and is given by
$$
\begin{array}{l} 
	X(s,t) = \left( \left(1\!-\!\frac{t^2}{2}\right)x(s)+\frac{t^2}{2}z(s),t\left(-x(s)\!+\!z(s)\right),-\frac{t^2}{2}x(s)+\left(1\!+\!\frac{t^2}{2}\right)z(s) \right) \\ \\
	\hspace{1.2cm} = \left(  \dfrac{u(s)+v(s)(1\!-\!t^2)}{2},-t\, v(s), \dfrac{u(s)-v(s)(1\!+\!t^2)}{2} \right), \ s\in I, \ t \in \R  .
\end{array}
$$

The three possibilities for the causal character of the axis of revolution, combined with the two causal characters of the generatrix curve and the plane where it lies in, is summarized in the following assertion (cf.\ \cite{dS21} or see the proof of Theorem \ref{Th:kmkp}): \textit{the hyperbolic rotational surfaces of second type are always timelike and the remaining rotational surfaces in $\L^3$ are spacelike (resp.\ timelike) if and only if their generatrix curves are spacelike (resp.\ timelike).} We also point out (see e.g.\ \cite[Section 2]{dS21} or the proof of Theorem \ref{Th:kmkp}) that the shape operator is always diagonalizable for surfaces of revolution in $\L^3$. 

It is clear that any rotational surface $ S_\gamma^l$ is completely determined by its generatrix curve $\gamma$.
Applying to $\gamma$ the results concerning its geometric linear momentum with respect to an axis ($l$ in this case) collected in Section \ref{SectMomentum}, taking into account that the uniqueness is modulo translations in the axis direction, we immediately conclude the following main result:

\begin{theorem}\label{Th:Kdetermine}
	Any non cylindrical rotational surface $S_\gamma^l$ in $\L^3$ is uniquely determined, up to translations along the axis of revolution $l$, by the geometric linear momentum $\mathcal K$ of its generatrix curve $\gamma$ with respect to $l$.
\end{theorem}

\begin{remark}\label{Re:Cylinders} We point out that we must exclude the right cylinders in Theorem \ref{Th:Kdetermine} because they correspond to the constant solutions ($x\equiv x_0$, $y\equiv y_0$, $z\equiv z_0$) of the separable differential equations that we solved in Section \ref{SectMomentum}. These cylinders  are usually called \textit{Lorentzian cylinders} if the axis is timelike and \textit{hyperbolic cylinders} if the axis is spacelike (see e.g.\ Example 3.3 in \cite{Lo14}). Following the above notation, we have the Lorentzian cylinder 
$X(s,t)=(x_0 \cos t, x_0 \sin t,s)$, $x_1^2+x_2^2=x_0^2$, which is timelike, the timelike hyperbolic cylinder $X(s,t)=(s,y_0 \cosh t, y_0 \sinh t)$, $x_2^2-x_3^2=y_0^2$, and the spacelike hyperbolic cylinder
$X(s,t)=(s,z_0 \sinh t, z_0 \cosh t)$, $x_3^2-x_2^2=z_0^2$. They all have constant principal curvatures, one of which is zero. See Figure \ref{fig:Cylinders}.
\end{remark}
\begin{figure}[h]
	\begin{center}
		\includegraphics[height=5cm]{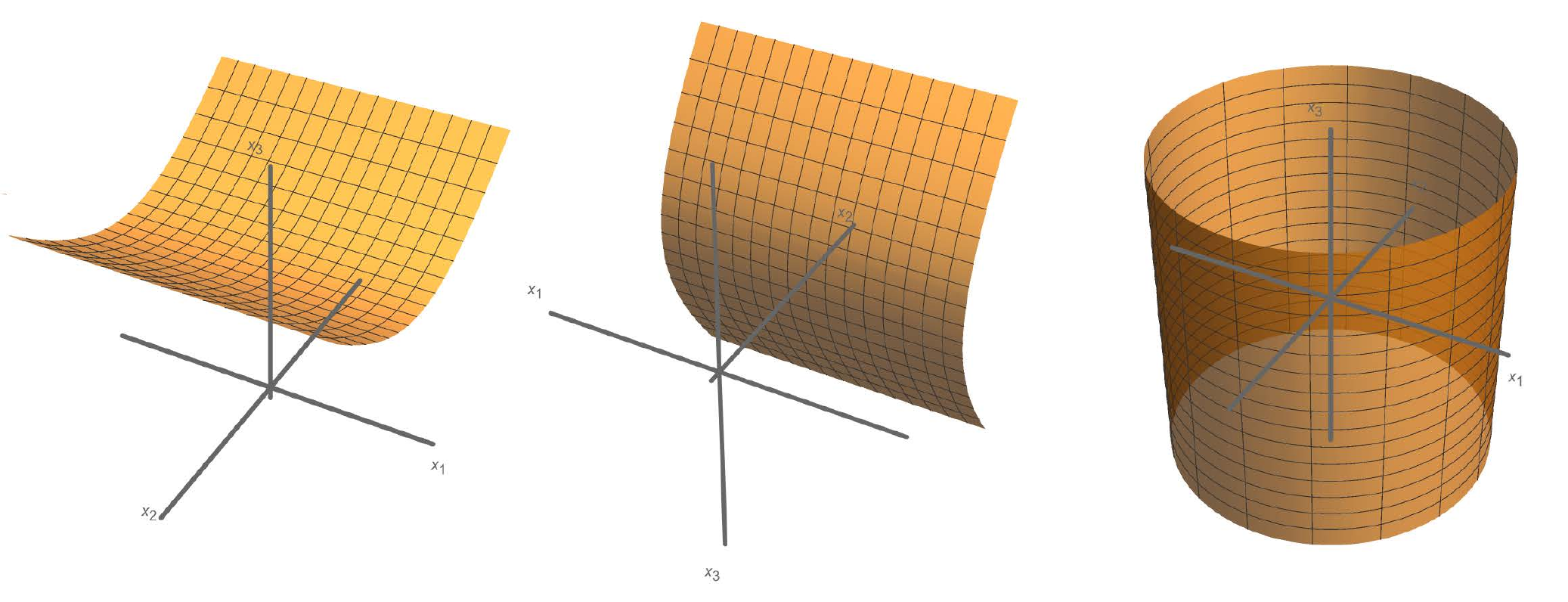}
		\caption{Cylinders $x_3^2-x_2^2=z_0^2$, $x_2^2-x_3^2=y_0^2$, $x_1^2+x_2^2=x_0^2$, in $\L^3$ ($x_0,y_0,z_0\in \R$).}
		\label{fig:Cylinders}
	\end{center}
\end{figure}
Making use of the description of the rotational surfaces of $\L^3$ given in Sections \ref{SectSxz_x}, \ref{SectSxy_x}, \ref{SectSxz_z} and \ref{SectSuv_u} and the study carried out in Section \ref{SectMomentum}, we are able to make fully explicit the statement of Theorem \ref{Th:Kdetermine} in the following key result of this paper:


\begin{corollary}\label{cor:Param} \phantom{salto}
	
\begin{enumerate} 
	
\item[\bf (1)] Any hyperbolic rotational surface of first type in $\L^3$  is locally congruent to the surface $  {S_{(x,z)}^{x_1}}$ given by
$$
 X_\epsilon(z,t)=\left(  x_\epsilon (z),z \, \sinh t, z \, \cosh t \right),
 \ t \in \R ,
$$
where the spacelike ($\epsilon =1$) or timelike ($\epsilon =-1$) generatrix curve $(x,z)$ is given by the graph
\begin{equation}\label{eq:xez}
x=x_\epsilon (z)=\int \!\frac{\mathcal K(z) \, dz}{\sqrt{\mathcal K(z)^2-\epsilon}},
\end{equation}
$z\in I=\{ z >0 : \mathcal K(z)^2-\epsilon > 0 \} $,
and $\mathcal K =\mathcal K(z)$ is the geometric linear momentum of $(x,z)$ with respect to the $x_1$-axis.
Its implicit equation is written as
\begin{equation}\label{eq:Imp_x1}
 x_3^2-x_2^2=z_\epsilon (x_1)^2,
\end{equation}
with $z=z_\epsilon (x)$ the inverse function of \eqref{eq:xez}.

\item[\bf (2)] Any hyperbolic rotational surface of second type in $\L^3$  is locally congruent to the surface $  {S_{(x,y)}^{x_1}}$ given by
$$
  X(y,t)=\left( x (y) , y \, \cosh t, y \, \sinh t \right),
 \ t \in \R ,
$$
where the Euclidean generatrix curve $(x,y)$ is given by the graph
\begin{equation}\label{eq:xey}
x= x (y)=\int \!\frac{\mathcal K(y) \, dy}{\sqrt{1-K(y)^2}},
\end{equation}
$ y\in I=\{ y>0 :  1-\mathcal K(y)^2 > 0 \} $,
and $\mathcal K =\mathcal K(y)$ is the geometric linear momentum of $(x,y)$ with respect to the $x_1$-axis.
Its implicit equation is written as
\begin{equation}\label{eq:Imp_x2}
 x_2^2-x_3^2=y (x_1)^2,
\end{equation}
with $y=y(x)$ the inverse function of \eqref{eq:xey}.

\item[\bf (3)] Any elliptic rotational surface in $\L^3$  is locally congruent to the surface $  {S_{(x,z)}^{x_3}}$ given by
$$
  X_\epsilon(x,t)=\left( x \, \cos t, x \, \sin t,  z_\epsilon (x) \right),
 \ t \in (0,2\pi),
$$
where the spacelike ($\epsilon =1$) or timelike ($\epsilon =-1$) generatrix curve $(x,z)$ is given by the graph
\begin{equation}\label{eq:zex}
z= z_\epsilon (x)=\int \!\frac{\mathcal K(x) \, dx}{\sqrt{\mathcal K(x)^2+\epsilon}},
\end{equation}
$ x\in I=\{ x >0 :  \mathcal K(x)^2  + \epsilon > 0 \} $,
and $\mathcal K =\mathcal K(x)$ is the geometric linear momentum of $(x,z)$ with respect to the $x_3$-axis.
Its implicit equation is written as
\begin{equation}\label{eq:Imp_z}
 x_1^2+x_2^2=x_\epsilon (x_3)^2,
\end{equation}
with $x=x_\epsilon (z)$ the inverse function of \eqref{eq:zex}.

\item[\bf (4)] Any parabolic rotational surface in $\L^3$ is locally congruent to the surface $  {S_{(u,v)}^u}$ given by
$$
\hspace{1cm} 	X_\epsilon(v,t)=\left( \frac{u_\epsilon (v)}{2} + \frac{v}{2}(1\!-\!t^2), -t \, v,  \frac{u_\epsilon (v)}{2} - \frac{v}{2}(1\!+\!t^2) \right),
\ t \in \R ,
$$
where the spacelike ($\epsilon =1$) or timelike ($\epsilon =-1$) generatrix curve $(u,v)$ is given by the graph
\begin{equation}\label{eq:uev}
u=u_\epsilon (v)=\epsilon \int \!\frac{dv}{\mathcal K(v)^2},
\end{equation}
$ v\in I=\{ v >0 : \mathcal K(v) > 0 \} $,
and $\mathcal K =\mathcal K(v)$ is the geometric linear momentum of $(u,v)$ with respect to the $u$-axis.
Its implicit equation is written as
\begin{equation}\label{eq:Imp_u}
 x_1^2+x_2^2-x_3^2=(x_1-x_3)u_\epsilon(x_1-x_3).
\end{equation}
\end{enumerate}
\end{corollary}

\begin{example}\label{ex:K0}
Suppose $\mathcal K \equiv 0$. Applying Corollary \ref{cor:Param}(1),
necessarily $\epsilon = -1$ for $S_{(x,z)}^{x_1}$ and we get  $X(z,t)=(x_0,z \sinh t, z \cosh t)$, i.e.\ the timelike plane $x_1=x_0$. Corollary \ref{cor:Param}(2) gives 
$X(y,t)=(x_0,y \cosh t, y \sinh t)$, i.e.\  the timelike plane $x_1=x_0$ again. Using Corollary \ref{cor:Param}(3), necessarily $\epsilon = 1$ for $S_{(x,z)}^{x_3}$ and we arrive at  $X(x,t)=(x \cos t, x \sin t, z_0)$, i.e.\ the spacelike plane $x_3=z_0$. Theorem \ref{Th:Kdetermine} implies that \textit{the (spacelike and timelike) planes are uniquely determined by a null geometric linear momentum.} 
\end{example}

\begin{remark}\label{Re:K^2}
	Using \eqref{eq:xez}, \eqref{eq:xey}, \eqref{eq:zex} and \eqref{eq:uev}, we point out that we can deduce new expressions for the square $\mathcal K^2$ of the geometric linear momentum  of a rotational surface in $\L^3$ in terms of the derivative of the graphs playing the role of generatrix curve at each case:
\begin{equation}\label{eq:K4}
	\mathcal K (z)^2= \frac{\epsilon \, x_\epsilon ' (z)^2}{x_\epsilon ' (z)^2 -1}, \
	\mathcal K (y)^2= \frac{x ' (y)^2}{1+x' (y)^2}, \
	\mathcal K (x)^2= \frac{\epsilon \, z_\epsilon ' (x)^2}{1-z_\epsilon ' (x)^2}, \
	\mathcal K (v)^2= \frac{\epsilon}{u_\epsilon ' (v)}.
\end{equation}
\end{remark}

The following result corroborates the fact that the geometric linear momentum of the generatrix curve with respect to the axis of revolution  fully determines a rotational surface in $\L^3$ (Theorem \ref{Th:Kdetermine}) by expressing the geometry of the surface (concretely, its principal curvatures) in terms of it. 

\begin{theorem}\label{Th:kmkp}
	The principal curvatures $k_{\rm m}$ (along meridians) and $k_{\rm p}$ (along parallels)  of a non cylindrical rotational surface $S_\gamma^l$ in $\L^3$ are given by
	\begin{equation}\label{eq:kmkp}
	   k_{\rm m}=\mathcal K' (r) , \quad k_{\rm p}=\frac{\mathcal K(r)}{r},
	\end{equation} 
where $\mathcal K=\mathcal K (r)$ is the geometric linear momentum of its generatrix curve $\gamma$ with respect to $l$, and $r$ is the (pseudo)distance from $\gamma$ to $l$.
\end{theorem}
\begin{remark}\label{Re:r}
Looking at Corollary \ref{cor:Param}, we point out that	$r=z$ for $S_{(x,z)}^{x_1}$, $r=y$ for $S_{(x,y)}^{x_1}$,  $r=x$ for $S_{(x,z)}^{x_3} $ and $r=v$ for $S_{(u,v)}^u $.
\end{remark}
\begin{proof}
Given a rotational surface $S_\gamma^l$ in $\L^3$, we control its local geometry by computing its first and second fundamental forms, $I$ and $I\!I$, following Section \ref{SectSurfaces} and using the $(s,t)$-coordinates given in Section \ref{SectRotational}. Both fundamental forms are expressed in terms of the geometry of the generatrix curve $\gamma$, so we will make use of the study carried out in Sections \ref{SectCurves} and \ref{SectMomentum}. As usual, $\kappa $ denotes the curvature of $\gamma $ and $\mathcal K$ its geometric linear momentum with respect to $l$.

In this way, starting at Section \ref{SectSxz_x}, a long straightforward computation for $S_{(x,z)}^{x_1}$ gives that $I\equiv \epsilon \, ds^2+z^2\, dt^2$, and  
$I\!I\equiv \epsilon \kappa \, ds^2+z \mathcal K \, dt^2$. Similarly, from Section \ref{SectSxy_x}, we obtain for $S_{(x,y)}^{x_1}$ that $I\equiv ds^2-y^2\, dt^2$, and  $I\!I\equiv - \kappa \, ds^2-y \mathcal K \, dt^2$. Using Section \ref{SectSxz_z}, we deduce for $S_{(x,z)}^{x_3}$ that $I\equiv \epsilon \, ds^2+x^2\, dt^2$, and  $I\!I\equiv \epsilon \kappa \, ds^2+x \mathcal K \, dt^2$. And, taking into account Section \ref{SectSuv_u}, we arrive for $S_{(u,v)}^u$ at $I\equiv \epsilon \, ds^2+v^2\, dt^2$, and  $I\!I\equiv -\epsilon \kappa \, ds^2+v \mathcal K \, dt^2$.

Hence we conclude the following expressions for the mean curvature $H$ and the Gauus curvature $K_{\rm G}$ of $S_\gamma^l$ in each case:
\begin{equation}\label{eq:Hrot}
	2H= \left\{ 
	\begin{array}{ll}
		-\epsilon \, (\kappa (z) + \mathcal K(z)/ z ) & {\rm for \ } S_{(x,z)}^{x_1} \\  
		-\kappa (y) + \mathcal K(y)/ y & {\rm for \ } S_{(x,y)}^{x_1} \\
		-\epsilon \, (\kappa (x) + \mathcal K(x)/ x ) & {\rm for \ } S_{(x,z)}^{x_3} \\ 
		-\epsilon \, (-\kappa (v) + \mathcal K(v)/ v ) & {\rm for \ } S_{(u,v)}^u 
	\end{array}
	\right.
\end{equation}
and
\begin{equation}\label{eq:KGrot}
	K_{\rm G} = \left\{ 
	\begin{array}{ll}
		-\epsilon \, \kappa (z) \, \mathcal K(z)/ z  & {\rm for \ } S_{(x,z)}^{x_1} \\  
		-\kappa (y) \, \mathcal K(y)/ y & {\rm for \ } S_{(x,y)}^{x_1} \\
		-\epsilon \, \kappa (x) \,  \mathcal K(x)/ x  & {\rm for \ } S_{(x,z)}^{x_3} \\ 
		\epsilon \, \kappa (v)  \, \mathcal K(v)/ v  & {\rm for \ } S_{(u,v)}^u 
	\end{array}
	\right.
\end{equation}
Looking at \eqref{eq:H}, \eqref{eq:Hrot} and \eqref{eq:KGrot}, and using that $\mathcal K'(z)=\kappa (z)$, $\mathcal K'(y)=-\kappa (y)$, $\mathcal K'(x)=\kappa (x)$ and $\mathcal K'(v)=-\kappa (v)$ (see Section \ref{SectMomentum}), we finally prove \eqref{eq:kmkp}.
%
%
%
%
\end{proof}

On the other hand, \textit{Weingarten surfaces} are defined by a functional relation between their principal curvatures, see for instance \cite{Ch45} or \cite{W61}.
For rotational Weingarten surfaces, we simply write $\Phi(k_{\text m}, k_{\text p})=0$, with $\Phi $ a differentiable function.
Thus, taking into account \eqref{eq:kmkp}, we easily deduce the following immediate consequence of Theorem \ref{Th:kmkp}.

\begin{corollary}\label{cor:edoW}
Any Weingarten relation $\Phi(k_{\text m}, k_{\text p})=0$ on a non cylindrical rotational surface $S_\gamma^l$ in $\L^3$ translates through \eqref{eq:kmkp} into a first-order ordinary differential equation $\hat \Phi (r,\mathcal K , \mathcal K')=0$ for the geometric linear momentum $\mathcal K$  of its generatrix curve $\gamma$ with respect to $l$ as a function of  the (pseudo)distance $r$ from $\gamma$ to $l$.
\end{corollary}	

\begin{example}\label{ex:Kcte}
	Suppose $k_{\rm m}=0$. Apart from the cylinders (see Remark \ref{Re:Cylinders}), using \eqref{eq:kmkp} we deduce 
	that $\mathcal K $ is constant. If $\mathcal K =0 $, from Example \ref{ex:K0} we arrive at the spacelike and timelike planes (which are the only rotational surfaces in $\L^3$ with $k_{\rm p}=0$). If $\mathcal K $ is a non-null constant, i.e.\ $\theta\equiv \theta_0 \in (0,\pi)$, $\varphi\equiv\varphi_0 \in \R$ and $\phi\equiv \phi_0 \in \R$ (see Section \ref{SectMomentum}), applying Corollary \ref{cor:Param} we reach the {\em Lorentzian rotational cones} collected in Table \ref{tab:Kcte}. See Figure \ref{fig:LorentzCones}.
	Theorem \ref{Th:Kdetermine} implies that \textit{the (spacelike and timelike) Lorentzian rotational cones are uniquely determined by a non-null constant geometric linear momentum.} 
\begin{table}[h]
		\caption{Lorentzian rotational  cones: $\mathcal{K}$ non-null constant.} 	\label{tab:Kcte}
	$
	\begin{array}{|c||c|c|}
		\hline
	S_\gamma^l & X_\epsilon (r,t) & \text{Implicit equation}\\
		\hline \hline
		
		  {S^{x_1}_{(x,z)}}, \epsilon = 1 &(\coth \varphi_0 \, z,z \sinh t, z \cosh t)  & x_3^2-x_2^2=\tanh^2 \varphi_0 \, x_1^2\\
		\hline
		
		  {S^{x_1}_{(x,z)}}, \epsilon = -1 & (\tanh \phi_0 \, z,z \sinh t, z \cosh t) & x_3^2-x_2^2=\coth^2 \phi_0 \, x_1^2 \\
		\hline \hline
		
		  {S^{x_1}_{(x,y)}} & (\cot \theta_0 \, y,y \cosh t, y \sinh t) & x_2^2-x_3^2=\tan^2 \theta_0 \, x_1^2 \\
		\hline \hline
		
		  {S^{x_3}_{(x,z)}}, \epsilon = 1 & (x \cos t, x \sin t, \tanh \varphi_0 \, x)   & x_1^2+x_2^2=\coth^2 \varphi_0 \, x_3^2 \\
		\hline
		  {{}S^{x_3}_{(x,z)}}, \epsilon = -1 & (x \cos t, x \sin t, \coth \phi_0 \, x) & x_1^2+x_2^2=\tanh^2 \phi_0 \, x_3^2 \\
		\hline \hline
		
		  {S^{u}_{(u,v)}}, \epsilon = 1 & \left( \frac{e^{2\varphi_0}v}{2} \!+\! \frac{v}{2}(1\!-\!t^2), -t \, v,  \frac{e^{2\varphi_0}v}{2} \!-\! \frac{v}{2}(1\!+\!t^2) \right)  & x_1^2\!+\!x_2^2\!-\!x_3^2=e^{2 \varphi_0} (x_1\!-\!x_3)^2\\
		\hline
		
		  {S^{u}_{(u,v)}}, \epsilon = -1 & \left( -\frac{e^{2\phi_0}v}{2} \!+\! \frac{v}{2}(1\!-\!t^2), -t \, v,  -\frac{e^{2\phi_0}v}{2} \!-\! \frac{v}{2}(1\!+\!t^2) \right)  & x_1^2\!+\!x_2^2\!-\!x_3^2=-e^{2 \phi_0} (x_1\!-\!x_3)^2 \\
		\hline
	\end{array}
	$
\end{table}

	\begin{figure}[h]
		\begin{center}
			\includegraphics[height=10cm]{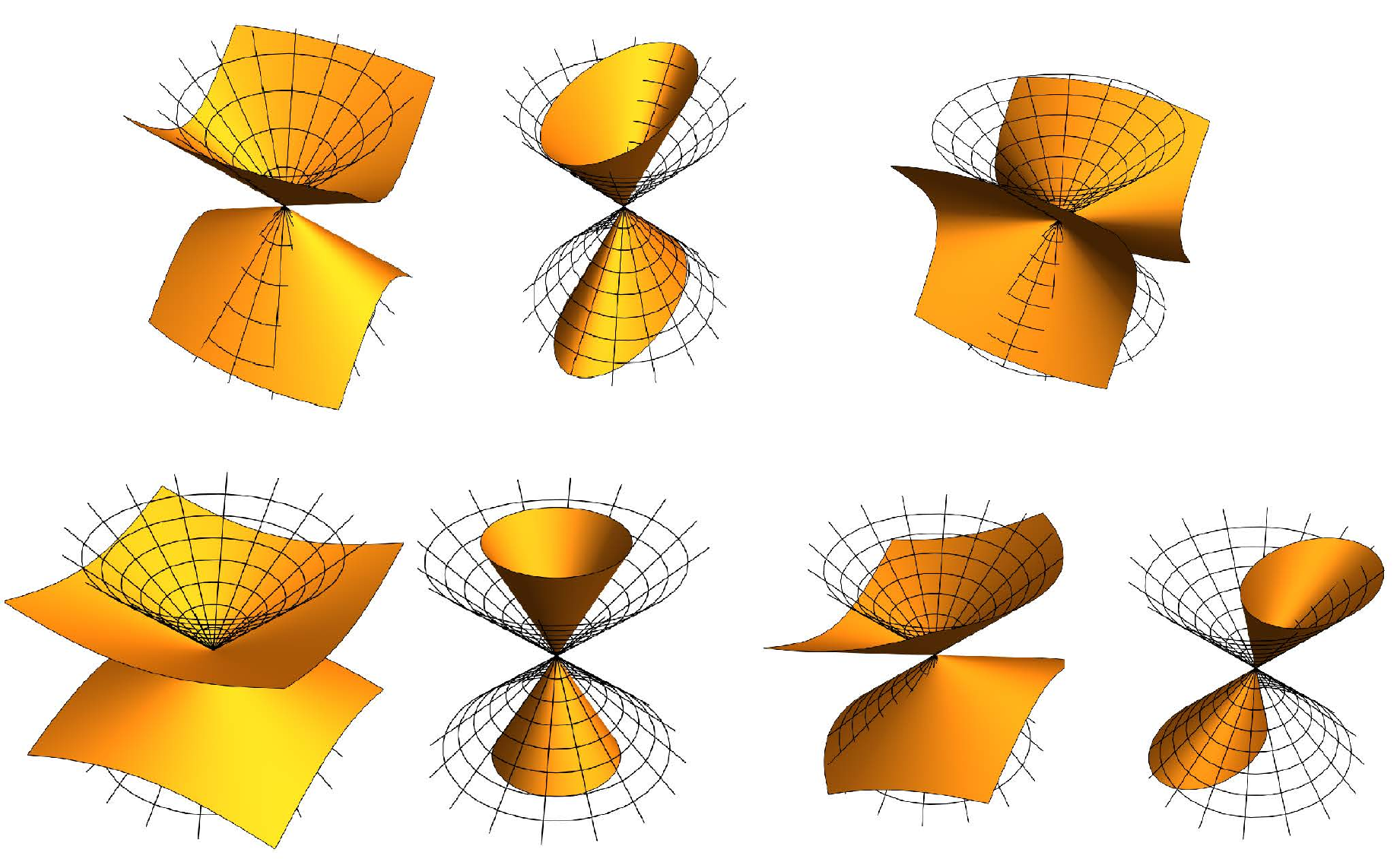}
			\caption{Lorentzian rotational cones in $\L^3$. First row: spacelike and timelike hyperbolic of first and second type; second row: spacelike and timelike elliptic and parabolic.}
			\label{fig:LorentzCones}
		\end{center}
	\end{figure}
\end{example}
\begin{example}\label{ex:K=r/R} 
	Suppose $k_{\rm m} \!=\! k_{\rm p}$. According to \eqref{eq:kmkp}, the corresponding o.d.e.\ is $\mathcal K ' = \mathcal K /r$. The trivial solution $\mathcal K \equiv 0$ leads to the planes (see Example \ref{ex:K0}). We can write the general solution as $\mathcal K (r)=r/R $, $R>0$. Then Corollary \ref{cor:Param} leads to the non flat umbilical surfaces collected in Table \ref{tab:umbilical}. 
	Hence we obtain three different parametrizations of the hyperbolic plane $\H^2_+(R)$ (see \cite{HN83}) and four distinct parametrizations of the de Sitter 2-space $\s^2_1(R)$. 
	Theorem \ref{Th:Kdetermine} implies that \textit{the  hyperbolic plane $\H^2_+(R)$ and the de Sitter 2-space $\s^2_1(R)$ are uniquely determined by the geometric linear momentum $\mathcal K (r)=r/R $, $R>0$}.
\begin{table}[h]
	\caption{Non flat umbilical surfaces.} 	\label{tab:umbilical}
	$
	\begin{array}{|c||c|c|c|}
		\cline{2-4}
		\multicolumn{1}{c|}{}& \K=\K(r) & X_\epsilon(r,t) & \text{Implicit equation}\\
		\hline \hline
		
			  {S^{x_1}_{(x,z)}} & \K(z)=z/R, &\multirow{2}{*}{$\left(\sqrt{z^2-\epsilon R^2}, z\sinh t, z\cosh t\right)$} & \multirow{2}{*}{$x_1^2+x_2^2-x_3^2=-\epsilon R^2$}\\
		(r=z) & R>0  &&\\
		\hline
			  {S^{x_1}_{(x,y)}} & \K(y)=y/R, & \multirow{2}{*}{$\left(-\sqrt{R^2-y^2}, z\cosh t, z\sinh t\right)$} & \multirow{2}{*}{$x_1^2+x_2^2-x_3^2=R^2$}\\
		(r=y) & R>0 &&\\
		\hline
			  {S^{x_3}_{(x,y)}} & \K(x)=x/R, & \multirow{2}{*}{$\left( x\cos t, x\sin t, \sqrt{x^2+\epsilon R^2}\right)$} & \multirow{2}{*}{$x_1^2+x_2^2-x_3^2=-\epsilon R^2$}\\
		(r=x) & R>0 &&\\
		\hline
			  {S^u_{(u,v)}} & \K(v)=v/R, & \multirow{2}{*}{$\left(-\frac{\epsilon R^2}{2v}+\frac{v}{2}(1\!-\!t^2) , -t \, v,   -\frac{\epsilon R^2}{2v} - \frac{v}{2}(1\!+\!t^2) \right)$} & \multirow{2}{*}{$x_1^2+x_2^2-x_3^2=-\epsilon R^2$}\\
		(r=v) & R>0 &&\\
		\hline
	\end{array}
	$	
\end{table}

\end{example}


\section{Linear Weingarten rotational surfaces in $\L^3$}\label{SectLinearWeingarten}


In this section we deal with rotational surfaces in $\L^3$ satisfying a linear relation between their principal curvatures. In Example \ref{ex:K=r/R} we discussed the case $k_{\rm m}=k_{\rm p}$. First we now analyze the equation $k_{\rm m}=-k_{\rm p}$, that is, those rotational surfaces such that $H=0$. Recall that 
spacelike surfaces with zero mean curvature are known as \textit{maximal} surfaces. Rotational maximal surfaces in $\L^3$ were classified by Kobayashi in \cite{Ko83}, and the classification of rotational surface with zero mean curvature in $\L^3$, including the timelike ones, was completed in \cite{VW90}. We provide a considerably shorter proof of both results as a direct application of Corollary \ref{cor:Param}.

\begin{theorem}\label{Th:H=0}
	\noindent 
	
	Every spacelike rotational surface with zero mean curvature in $\L^3$ is locally congruent to one of the following (see Figure \ref{fig:H0spacelike}):
	\begin{itemize}
		\item[(i)] the spacelike plane $x_3=z_0$, $z_0\in \R$;
		\item[(ii)] the catenoid of first kind $x_1^2+x_2^2=a^2 \sinh^2 (x_3/a)$, $a>0$;
		\item[(iii)] the catenoid of second kind $x_3^2-x_2^2=a^2 \sin^2 (x_1/a)$, $a>0$;
		\item[(iv)]  the Enneper surface of second kind $x_1^2+x_2^2-x_3^2\!=\!a^2(x_1\!-\!x_3)^4$, $a>0$.
	\end{itemize}
	\noindent 
	
	Every timelike rotational surface with zero mean curvature in $\L^3$ is locally congruent to one of the following (see Figure \ref{fig:H0timelike}):
	\begin{itemize}
		\item[(v)] the timelike plane $x_1=x_0$, $x_0\in \R$; 
		\item[(vi)] the catenoid of third kind $x_1^2+x_2^2=a^2 \sin^2 (x_3/a)$, $a>0$;
		\item[(vii)] the catenoid of fourth kind $x_3^2-x_2^2=a^2 \sinh^2 (x_1/a)$, $a>0$;
		\item[(viii)] the catenoid of fifth kind $x_2^2-x_3^2=a^2 \cosh^2 (x_1/a)$, $a>0$;
		\item[(ix)]  the Enneper surface of third kind $x_1^2+x_2^2-x_3^2\!=\!-a^2(x_1\!-\!x_3)^4$, $a>0$.
	\end{itemize}
\end{theorem}
\begin{figure}[h]
	\begin{center}
		\includegraphics[height=5cm]{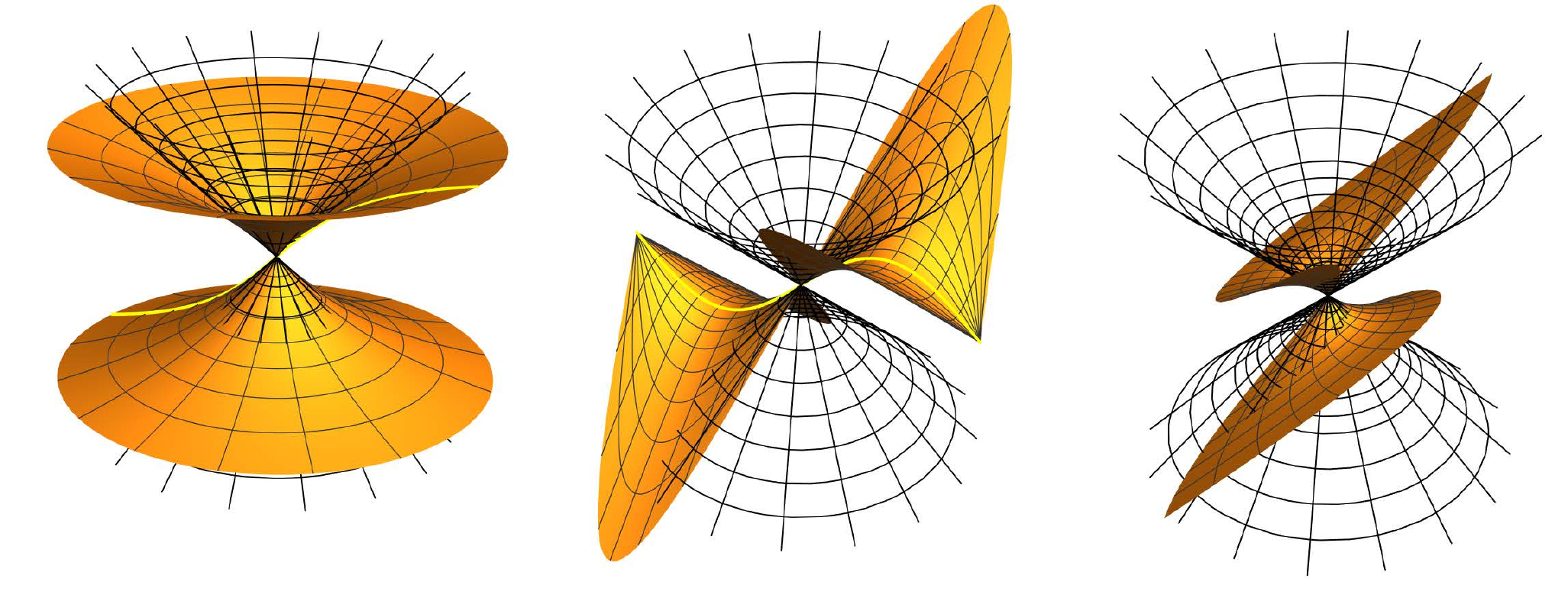}
		\caption{Spacelike rotational surface with zero mean curvature in $\L^3$: the catenoid of first kind, the catenoid of second kind, and the Enneper surface of second kind.}
		\label{fig:H0spacelike}
	\end{center}
\end{figure}

\begin{figure}[h]
	\begin{center}
		\includegraphics[height=10cm,width=10cm]{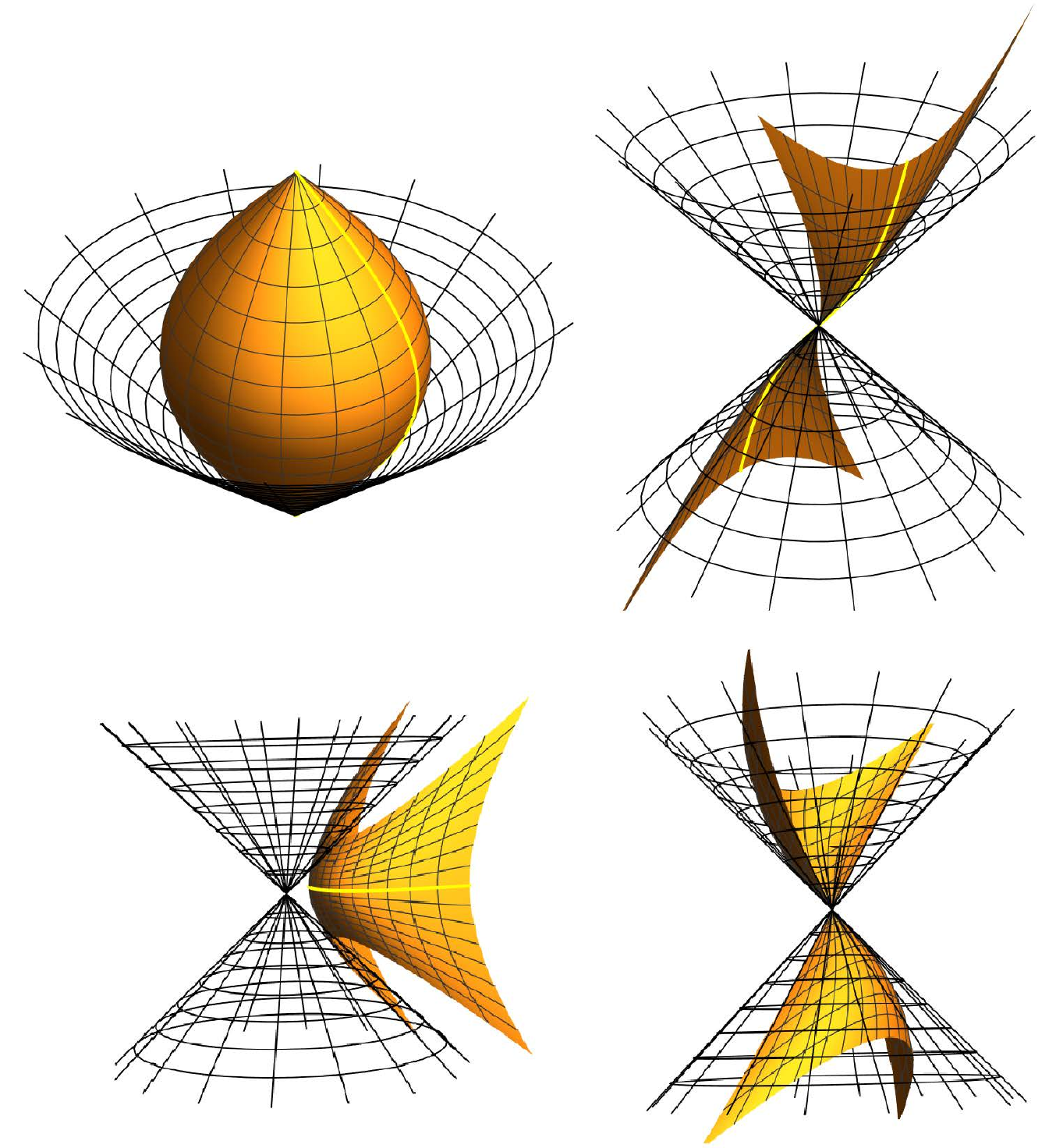}
		\caption{Timelike rotational surface with zero mean curvature in $\L^3$: the catenoid of third kind, the catenoid of fourth kind, the catenoid of fifth kind, and the Enneper surface of third kind.}
		\label{fig:H0timelike}
	\end{center}
\end{figure}
\begin{proof}
	According to Theorem \ref{Th:kmkp}, the hypothesis $H=0$ translates into the linear o.d.e.\ $\mathcal K'(r)+\mathcal K (r)/r=0$. 
	Its general solution is given by $\mathcal K (r)=a/r$, $a\geq 0$. Its constant solution $\mathcal K \equiv 0$ (see Example \ref{ex:K0}) leads to parts (i) and (v). 
	
	Taking into account Remark \ref{Re:r}, we use Corollary \ref{cor:Param} and get from \eqref{eq:xez} that
	$x_1(z)= a \arcsin (z/a)$ and $x_{-1}(z)= a \arcsinh (z/a)$ in $S_{(x,z)}^{x_1}$. Using this in \eqref{eq:Imp_x1}, we reach parts (iii) and (vii).
	In the same way, from \eqref{eq:xey} we have that 
	$x(y)= a \arccosh (y/a)$ in $S_{(x,y)}^{x_1}$, and \eqref{eq:Imp_x2} leads to part (viii).
	Similarly, from \eqref{eq:zex} we obtain that
	$z_1(x)= a \arcsinh (x/a)$ and $z_{-1}(z)= a \arcsin (x/a)$ in $S_{(x,z)}^{x_3}$, and \eqref{eq:Imp_z} implies parts (ii) and (vi).
	Finally, we deduce from \eqref{eq:uev} that
	$u_1(v)=v^3/3a^2$ and $u_{-1}=-v^3/3a^2$ in $S_{(u,v)}^u$. Renaming $a^2\equiv1/3a^2$, \eqref{eq:Imp_u} conclude parts (iv) and (ix).
\end{proof}

Our next intention is an ostensible generalization of Theorem \ref{Th:H=0} because we aspire to classify the rotational surfaces in $\L ^3$ satisfying the linear Weingarten relation $k_{\text m} = q\,  k_{\text p}$, $q\neq 0$. Note then that Theorem \ref{Th:H=0} would simply be the case $q=-1$. With that in mind and inspired by Definition 2.5 in \cite{CC24}, we begin by introducing a family of rotational surfaces in $\L^3$ which plays the role of the classical Hopf surfaces in $\E^3$ (see \cite{CC24}, \cite{H51} and \cite{K15}). 

\begin{definition}[Lorentzian Hopf surfaces]\label{def:Hopf}
	
	Consider in $\L^3$ the following families of rotational surfaces $S_{\gamma_q^a}^l$, $q\neq 0$, $a>0$, defined according to their following generatrix curves:

	\begin{enumerate}
		\item[(I)]  \textit{Hyperbolic Lorentzian Hopf surfaces of first type, $S_{\gamma_q^a}^{x_1}$:}
	\begin{enumerate}
		\item[(I-T)] with $\gamma_q^a=(x,z)$  given by
		$$
		x= \frac{a}{q}\int_0^t \sinh ^{1/q} v \, dv, \ 	z= a \sinh ^{1/q} t,  \ t >0, \ a>0. 
		$$
		\item[(I-S)] with $\gamma_q^a=(x,z)$ given by
		$$
		x= \frac{a}{q}\int_0^t \cosh ^{1/q} v \, dv, \ 	z= a \cosh ^{1/q} t, \ t\in \R,   \ a>0. 
		$$
	\end{enumerate}
	\item[(II)] \textit{ Hyperbolic Lorentzian Hopf surfaces of second type, $S_{\gamma_q^a}^{x_1}$:}
	
	 with $\gamma_q^a=(x,y)$ is given by
	$$
	 x= -\frac{a}{q}\int_0^t \cos ^{1/q} v \, dv, \ y= a \cos ^{1/q} t, \ t\in (-\pi/2,\pi/2), \ a>0. 
	$$
	\item[(III)]  \textit{Elliptic Lorentzian Hopf surfaces, $S_{\gamma_q^a}^{x_3}$:}
	\begin{enumerate}
		\item[(III-T)] with $\gamma_q^a=(x,z)$ given by
		$$
		x= a \cosh ^{1/q} t, \ z= \frac{a}{q}\int_0^t \cosh ^{1/q} v \, dv, \ t\in \R, \  \ a>0. 
		$$
		\item[(III-S)] with $\gamma_q^a=(x,z)$ given by
		$$
		x= a \sinh ^{1/q} t, \ z= \frac{a}{q}\int_0^t \sinh ^{1/q} v \, dv, \ t >0, \  	 a>0. 
		$$
	\end{enumerate}

	\item[(IV)]  \textit{Parabolic Lorentzian Hopf surfaces, $S_{\gamma_q^a}^{u}$:} 
	 \begin{enumerate}
	 	\item $q\neq 1/2$: with $\gamma_q^a=(u,v)$ given by $$u=\frac{\epsilon a^{2q}}{1-2q}\, v^{1-2q}, \ v>0, \ a>0 \  (\epsilon = \pm 1). $$
	 	\item $q =  1/2$: with $\gamma_{1/2}^a=(u,v)$ given by $$u=\epsilon a^{2q}\, \ln v, \ v>0, \ a>0 \  (\epsilon = \pm 1). $$
	 \end{enumerate}
	\end{enumerate}
\end{definition}
In parts (I), (II) and (III), (-T) means timelike surfaces and (-S) mean spacelike surfaces; in part (IV), $\epsilon =1$ (resp.\ $\epsilon =-1$) corresponds to spacelike (resp.\ timelike) surfaces, as usual. We call them {\em Lorentzian Hopf surfaces} because the Euclidean counterparts were first introduced by Heinz Hopf  \cite{H51} in 1951. It can be verified that when $q=1$, the non flat umbilical surfaces are recovered (see Example \ref{ex:K=r/R}), and when $q=-1$, those ones with zero mean curvature (see Theorem \ref{Th:H=0}). See Figure \ref{fig:HopfLorentz}.

\begin{figure}[h]
	\begin{center}
		\includegraphics[height=12cm]{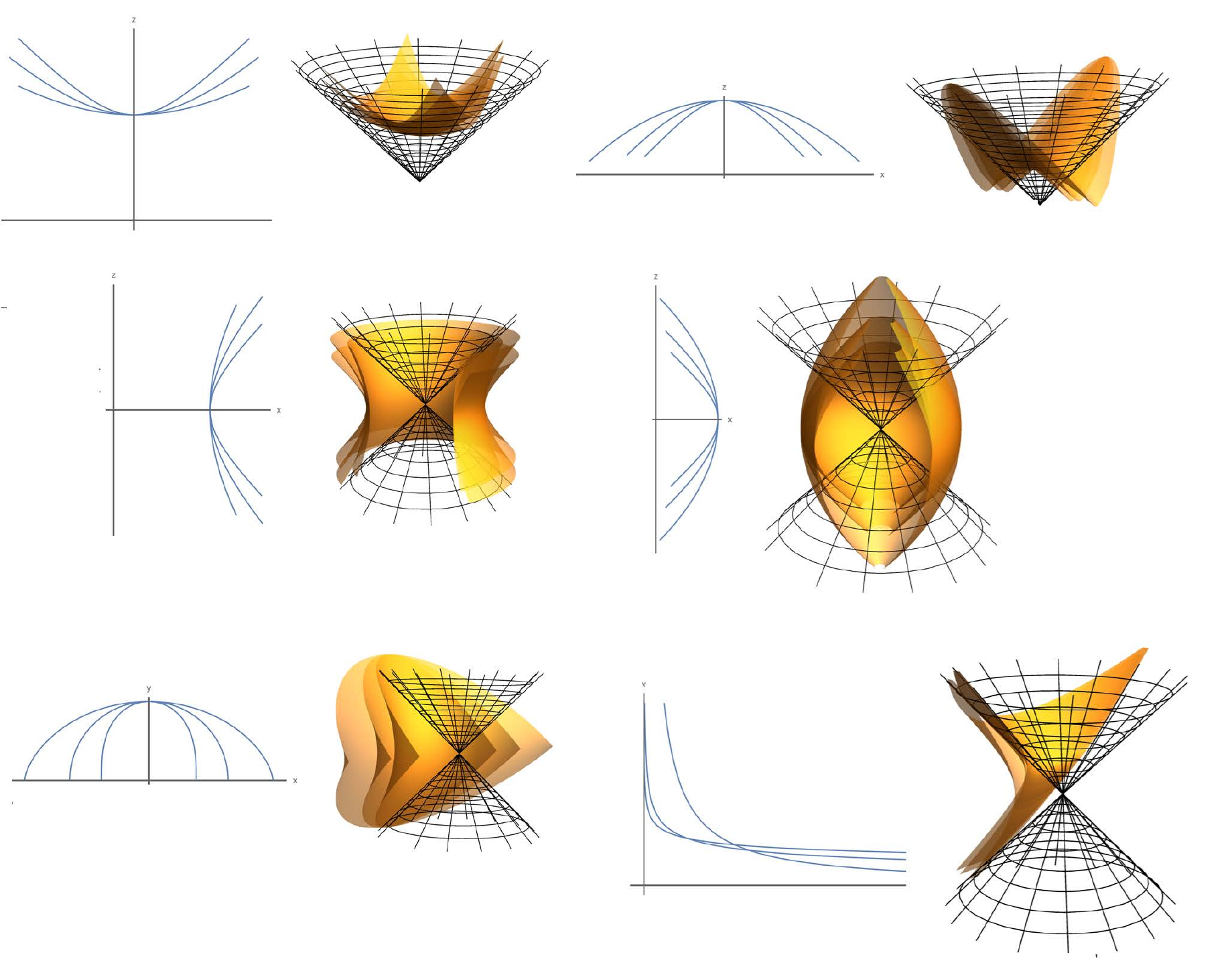}
		\caption{Some Lorentzian Hopf surfaces in $\L^3$ and their generatrix curves.
			First row: spacelike hyperbolic of first type, $q>0$ and $q<0$;
			second row: timelike elliptic $q>0$ and $q<0$; 
		third row: hyperbolic of second type, $q>0$, and timelike parabolic $q\neq 1/2$. }
		\label{fig:HopfLorentz}
	\end{center}
\end{figure}

Using Section \ref{SectMomentum}, Remark \ref{Re:K^2} and Remark \ref{Re:r}, it is a long straightforward computation to check that the linear geometric momentum of $\gamma_q^a$, $q\neq 0$, $a>0$, is given by
\begin{equation}\label{eq:K Hopf}
	\mathcal K (r)= \frac{r^q}{a^q}.
\end{equation}

In the next result, we determine that the rotational surfaces in $\L^3$ such that both principal curvatures are proportional at each point are precisely the Lorentzian Hopf surfaces introduced in Definition \ref{def:Hopf}.

\begin{theorem}\label{Th:Hopf LW}
	The only rotational surfaces satisfying the linear Weingarten relation $k_{\text m} = q\,  k_{\text p}$, $q\neq 0$, are the spacelike and timelike planes and the Lorentzian Hopf surfaces $S_{\gamma_q^a}^l$, $a>0$, described in Definition \ref{def:Hopf}.
\end{theorem}

\begin{proof}
	First, the spacelike and timelike plane trivially satisfies $k_{\text m} = q\,  k_{\text p}$, for any $q\neq 0$. Using \eqref{eq:K Hopf} in \eqref{eq:kmkp}, it is easy to check that $S_{\gamma_q^a}^l$ verifies precisely $k_{\text m} = q\,  k_{\text p}$.
	
	On the other hand, using Theorem \ref{Th:kmkp},
	the~linear Weingarten relation $k_{\text m}=q\,  k_{\text p}$ translates into the separable~o.d.e.
	$$\mathcal K ' (r)=q\, \mathcal K(r)/r. $$
	Its constant solution $\mathcal K \equiv 0$ leads to the spacelike and timelike planes (see Example \ref{ex:K0}).
	And its non-constant solution is given by $\mathcal K (r)= c \, r^q$, $c>0$. Taking $a=1/c^{1/q}$, we arrive at \eqref{eq:K Hopf} and then Theorem \ref{Th:Kdetermine} finishes the proof.
\end{proof}

\section{Quadratic Weingarten rotational surfaces in $\L^3$}\label{SectQuadraticWeingarten}


In this section, we are interested in the rotational Weingarten surfaces of $\L^3$ whose principal curvatures are related by the quadratic equation  $k_{\rm m}=\mu \, k_{\rm p}^2$, $\mu \neq 0$. Applying Theorem \ref{Th:kmkp} and taking into account \eqref{eq:kmkp}, we arrive at the differential equation
\begin{equation}\label{eq:OdeQuadratic}
	\mathcal{K}'(r)=\mu\, \mathcal{K}(r)^2/r^2,
\end{equation}
whose solutions determine the corresponding surfaces through Corollary \ref{cor:Param}.

Notice that the constant trivial solution $\mathcal{K}\equiv 0$ provides the spacelike and timelike planes in view of Example \ref{ex:K0}. The general solution of \eqref{eq:OdeQuadratic} is given by
\begin{equation}\label{eq:momemtum}
	\mathcal{K}_{\mu, c}(r)=\frac{r}{\mu+c r},\ c\in\mathbb R, \ \mu+c r\neq 0.
\end{equation}

Since $\mathcal{K}_{-\mu, -c}(r)=-\mathcal{K}_{\mu, c}(r)$, we may assume that $\mu>0$ in the following. If $c=0$, the geometric linear momentum \eqref{eq:momemtum} turns into $\mathcal{K}_{\mu,0}(r)=r/\mu$, and therefore the surface is locally congruent to the  hyperbolic plane $\H^2_+(\mu)$ if it is spacelike or to the de Sitter 2-space $\s^2_1(\mu)$ it is timelike (see Example \ref{ex:K=r/R}).

Using Corollary \ref{cor:Param}, we deduce that the desired rotational surfaces are generated by the following graphs:
\begin{enumerate}
	\item Hyperbolic of first type case:
	\begin{equation}
		\label{eq:graphH1}
		x=x_\epsilon^{\mu,c}(z)=  \int \dfrac{ z\, dz}{\sqrt{P_\epsilon^{\mu,c}(z)}}, \ P_\epsilon^{\mu,c}(z)=(1-\epsilon c^2)z^2-2\epsilon c \mu z-\epsilon \mu^2 >0 .
	\end{equation}
	\item Hyperbolic of second type case:
	\begin{equation}
		\label{eq:graphH2}
		x=x^{\mu,c}(y)=  \int \dfrac{ y\, dy}{\sqrt{Q^{\mu,c}(y)}}, \ Q^{\mu,c}(y)=(c^2-1)y^2+2c \mu y+ \mu^2>0.
	\end{equation}
	\item Elliptic case:
	\begin{equation}
		\label{eq:graphE}
		z=z_\epsilon^{\mu,c}(x)=  \int \dfrac{ x\, dx}{\sqrt{R_\epsilon^{\mu,c}(x)}}, \ R_\epsilon^{\mu,c}(x)=(1+\epsilon c^2)x^2+2\epsilon c \mu x+\epsilon \mu^2 >0.
	\end{equation}
	\item Parabolic case:
	\begin{equation}
		\label{eq:graphP}
		u=u^{\mu,c}_\epsilon(v)= \epsilon \int \dfrac{(\mu +cv)^2}{v^2}dv, \ v>0.
	\end{equation}
\end{enumerate}

We point out that $x_\epsilon^{\mu,-c}(-z)=x_\epsilon^{\mu,c}(z)$, $x^{\mu,-c}(-y)=x^{\mu,c}(y)$ and $z_\epsilon^{\mu,-c}(-x)=z_\epsilon^{\mu,c}(x)$; hence we may also assume that $c > 0$ in \eqref{eq:graphH1}, \eqref{eq:graphH2} and \eqref{eq:graphE}.
Furthermore, $R_\epsilon^{\mu,c}(x)=P_{-\epsilon}^{\mu,c}(x)$ and hence $z_{\epsilon}^{\mu,c}(x)=x_{-\epsilon}^{\mu,c}(x)$.

A long but straightforward computation provides the following result.
\begin{proposition}\label{lem:graphs}
	The explicit expressions of the graphs described in equations \eqref{eq:graphH1}, \eqref{eq:graphH2}, \eqref{eq:graphE}, and \eqref{eq:graphP}, depending on $c > 0$ and considering $\mu>0$, are depicted in the following list:
	\begin{enumerate}[(1)]
		\item Hyperbolic of first type:
		\begin{enumerate}
			\item Spacelike generatrix, $\epsilon=1$:
			\begin{enumerate}[(i)]
				\item If $0<c<1$:  
				$$x_1^{\mu,c}(z) = \frac{\sqrt{P_1^{\mu,c}(z)}}{1 - c^2} + \frac{c \mu}{(1 - c^2)^{3/2}}
				\log\left(z + \frac{c \mu}{ c^2-1} + \frac{\sqrt{P_1^{\mu,c}(z)}}{\sqrt{1 - c^2}}\right)
				, \text{ with $z \leq\frac{-\mu}{1+c}$ or $z\geq \frac{-\mu}{-1+c}$.}$$
				\item If $c=1$: $$\displaystyle x_{1}^{\mu,1}(z)=\frac{(-z + \mu) \sqrt{-\mu (2 z + \mu)}}{3 \mu}, \text{ with } z\leq \mu/2.$$
				\item If $c>1$: $$\displaystyle x_1^{\mu,c}(z)= -\frac{\sqrt{P_1^{\mu,c}(z)}}{c^2 - 1} + \frac{c \mu \arccos\left( \frac{(c^2 - 1) z}{\mu}+c\right)}{(c^2 - 1)^{3/2}}, \text{ with } \frac{-\mu}{1+c}\leq z\leq \frac{-\mu}{-1+c}.$$
			\end{enumerate}
			\item Timelike generatrix, $\epsilon=-1$:
			\[
			x_{-1}^{\mu,c}(z)= \frac{\sqrt{P_{-1}^{\mu,c}(z)}}{1 + c^2} - \frac{c \mu}{(1 + c^2)^{3/2}} \log\left(z + \frac{c \mu}{1 + c^2} + \frac{\sqrt{P_{-1}^{\mu,c}(z)}}{\sqrt{1 + c^2}}\right), \ z\in\R.
			\]
		\end{enumerate}
		\item Hyperbolic of second type:
		\begin{enumerate}[(i)]
			\item If $0<c<1$: 
			$$ x^{\mu,c}(y)=\frac{\sqrt{Q^{\mu,c}(y)}}{c^2 - 1} + \frac{c \mu \left(\pi - \arccos\left(\frac{(1 - c^2)y}{\mu} - c\right)\right)}{(1 - c^2)^{3/2}},\text{ with }  \frac{-\mu}{1+c}\leq y\leq \frac{-\mu}{-1+c}.$$
			\item If $c=1$: 
			$$\displaystyle x^{\mu,1}(y)= \frac{(y - \mu) \sqrt{\mu (2 y + \mu)}}{3 \mu}, \text{ with }  y\geq\mu/2.$$
			\item If $c>1$: 
			$$ x^{\mu,c}(y)=  \frac{\sqrt{Q^{\mu,c}(y)}}{c^2 - 1} - \frac{c \mu}{(c^2 - 1)^{3/2}}
			\log\left(y + \frac{c \mu}{c^2 - 1} + \frac{\sqrt{Q^{\mu,c}(y)}}{\sqrt{c^2 - 1}}\right)
			, \text{ with $ y \leq\frac{-\mu}{1+c} $ or $ y\geq \frac{-\mu}{-1+c}$.}$$
		\end{enumerate}
		\item Elliptic: $z_{\epsilon}^{\mu,c}(x)=x_{-\epsilon}^{\mu,c}(x)$, $\epsilon =\pm 1$, with $x$ varying as in case (1).
		\item Parabolic: $\displaystyle u_{\epsilon}^{\mu,c}(v)=\epsilon \left(c^2 v - \frac{\mu^2}{v} + 2 c \mu \log(v)\right)$, $\epsilon =\pm 1$, $v>0$.
	\end{enumerate}
\end{proposition}

Therefore, as a consequence of Corollary \ref{cor:edoW} and Corollary \ref{cor:Param}, we have proved the following local classification result for quadratic rotational Weingarten surfaces:
\begin{theorem}\label{Th:Wquadratic}
	The rotational surfaces in $\L^3$ whose principal curvatures satisfy $k_m=\mu \,  k_p^2$, $\mu\neq0$, are locally congruent to the spacelike and timelike planes, the hyperbolic plane and De Sitter 2-space of radius $|\mu|$, and the rotational surfaces generated by the graphs described in Proposition \ref{lem:graphs}.
\end{theorem}

\begin{figure}[h]
	\begin{center}
			\includegraphics[height=8cm]{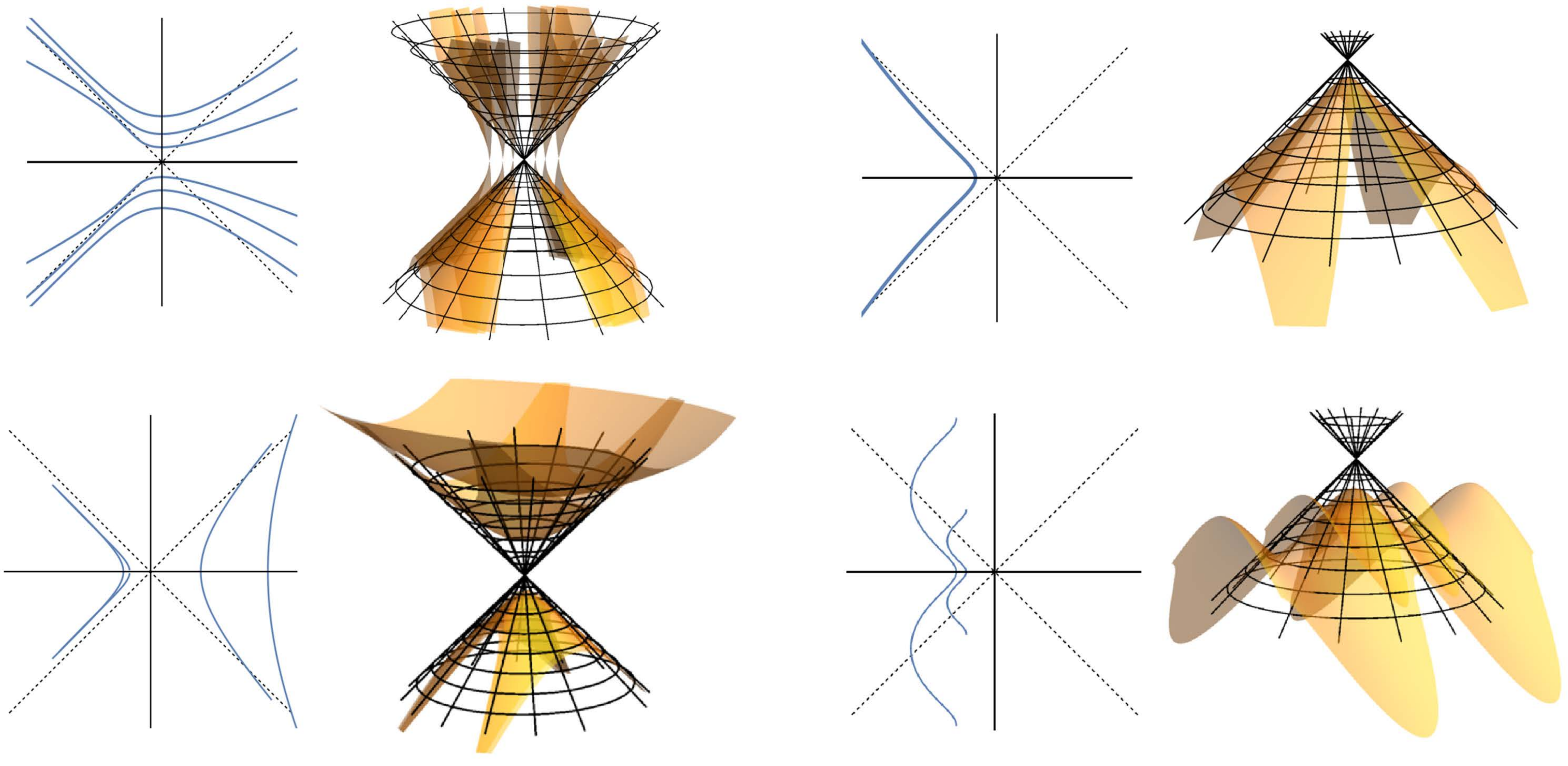} \hspace{1cm}
		\end{center}
	\caption{Curves $x_\epsilon^{\mu,c}(z)$ and the corresponding rotational hyperbolic surfaces of first type for $\mu=1$, $\epsilon=1$ and $0<c<1$ (top-left), $c=1$ (top-right), $c>1$ (down-left); and $\epsilon=-1$ and some $c>0$ (down-right). } \label{fig:H1}
\end{figure}

\begin{figure}[h]
	\begin{center}
		\includegraphics[height=8cm]{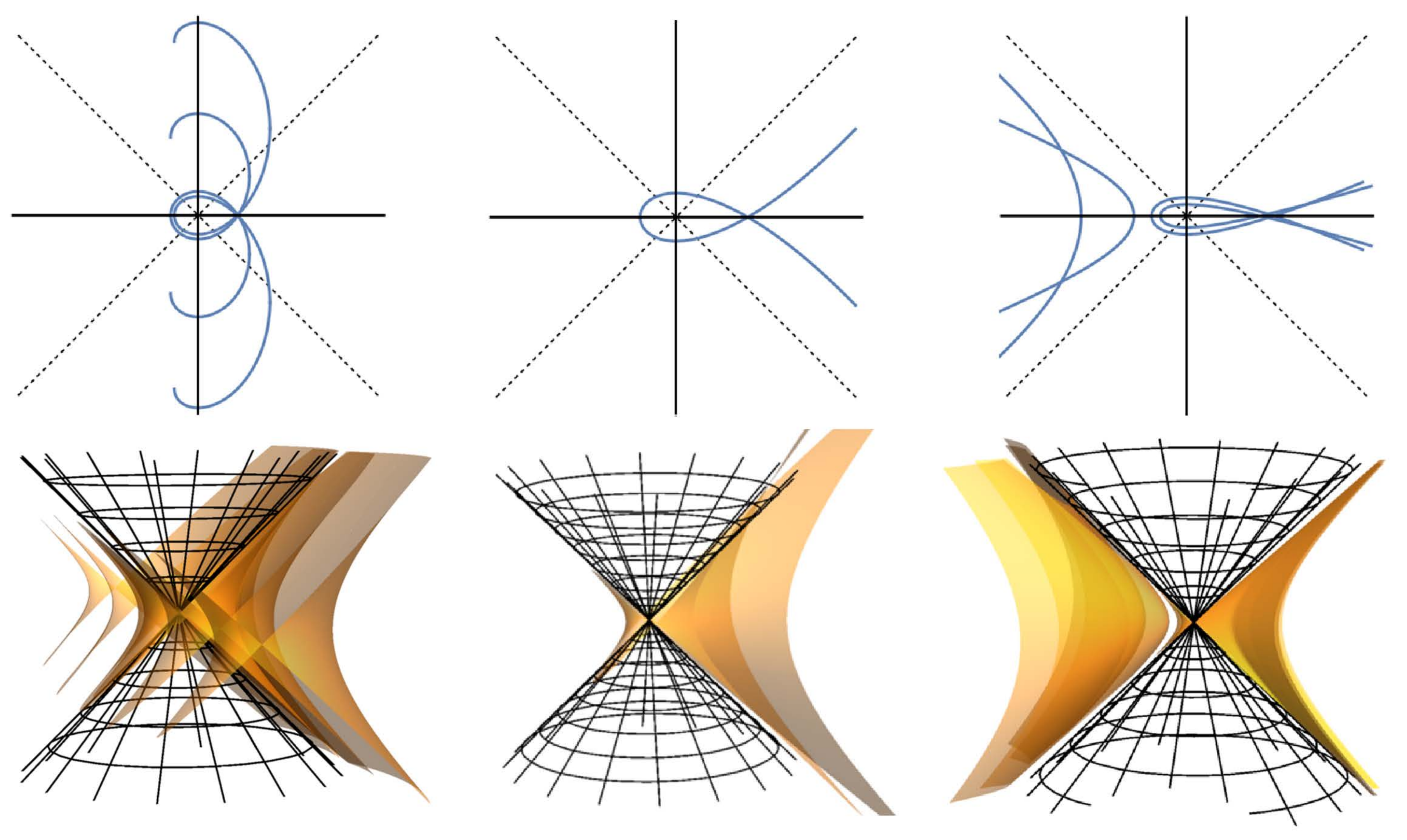} \hspace{1cm}
	\end{center}
	\caption{Curves $x^{\mu,c}(y)$ and the corresponding rotational hyperbolic surfaces of second type for $\mu=1$, and $0<c<1$ (left), $c=1$ (middle), and $c>1$ (right).} \label{fig:H2}
\end{figure}

\begin{figure}[h]
	\begin{center}
		\includegraphics[height=8cm]{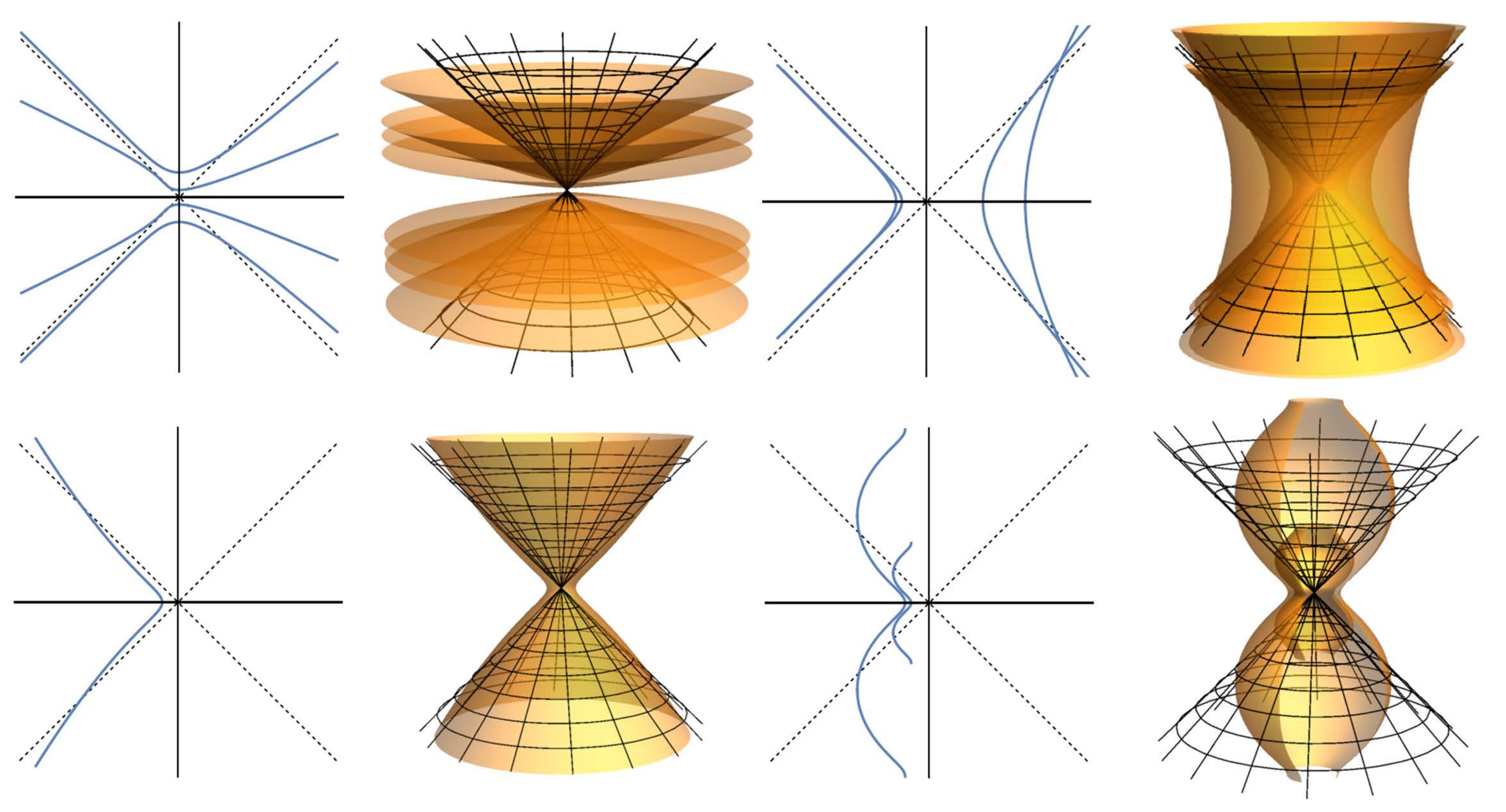} \hspace{1cm}
	\end{center}
	\caption{Curves $z_\epsilon^{\mu,c}(x)$ and the corresponding rotational elliptic surfaces for $\mu=1$, $\epsilon=1$ and some $c>0$ (top-left), $\epsilon=-1$ and $0<c<1$ (top-right), $c=1$ (down-left), and $c>1$ (down-right).} \label{fig:E}
\end{figure}

\begin{figure}[h]
	\begin{center}
		\includegraphics[height=4cm]{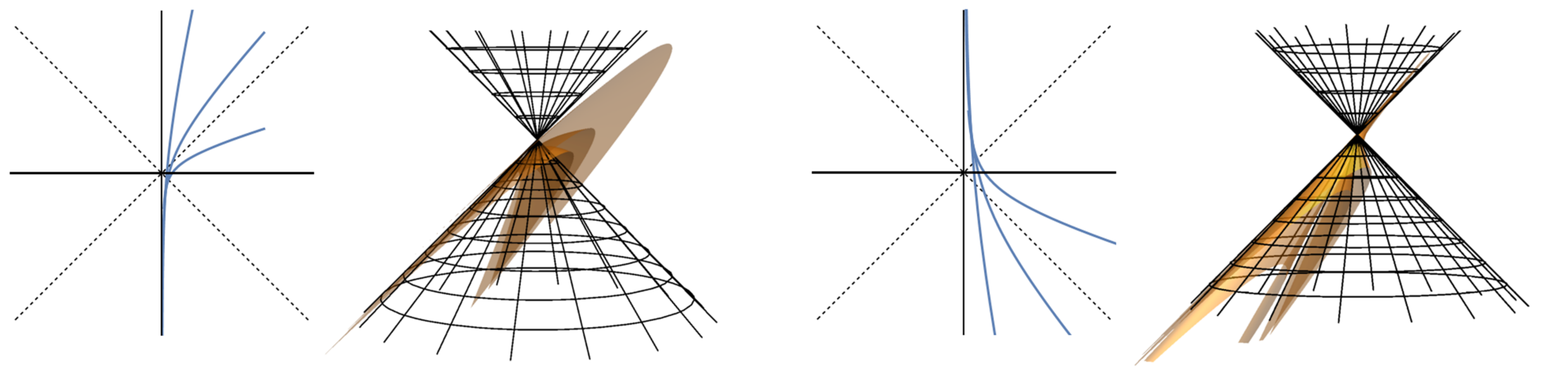} \hspace{1cm}
	\end{center}
	\caption{Curves $u_\epsilon^{\mu,c}(v)$ and the corresponding rotational parabolic surfaces for  $\mu=1$, $\epsilon= 1$ (left) or $\epsilon=-1$ (right), and some values of $c>0$.} \label{fig:P}
\end{figure}


\section{Cubic Weingarten rotational surfaces in $\L^3$}\label{SectLCubicWeingarten}

The aim of this section is to study the rotational surfaces in $\L^3$ satisfying the cubic Weingarten equation $k_{\text m}=\mu\,  k_{\text p}^3$, with $\mu \neq 0$.

\subsection{Non-degenerate quadric surfaces of revolution in $\mathbb L^3$}\label{SectQuadric}

 Inspired by \cite{CC22} and \cite{CC23}, we look first for the Lorentzian versions of the non-degenerate quadric surfaces of revolution in $\mathbb E^3$. They can be rewritten in terms of their circular orbits in the following way:
\begin{enumerate}
	\item ellipsoid of revolution $\frac{x^2+y^2}{a^2}+\frac{z^2}{b^2}=1$: $$x^2+y^2=\alpha^2-\beta^2 \, z^2, $$
	\item one-sheet hyperboloid of revolution $\frac{x^2+y^2}{a^2}-\frac{z^2}{b^2}=1$: $$x^2+y^2=\alpha^2+\beta^2 \, z^2, $$
	\item two-sheets hyperboloid of revolution $-\frac{x^2+y^2}{a^2}+\frac{z^2}{b^2}=1$: $$ x^2+y^2=-\alpha^2+\beta^2 \, z^2, $$
	\item paraboloid of revolution $z=\frac{x^2+y^2}{2a}$: $$x^2+y^2=\delta \, z, $$
\end{enumerate}
where $a,b>0$ and $\alpha=a>0$, $\beta=a/b>0$, $\delta=2a>0$.

As far as we know, there is hardly any literature about the non-degenerate quadric surfaces of revolution in $\mathbb L^3$. It seems natural then to follow the above idea in order to define the \textit{non-degenerate quadric surfaces of revolution} in $\mathbb L^3$  taking into account the possible nature of the orbits in the three types of rotations in $\L^3$. If we also introduce the parameters $a=\alpha >0$, $b=\alpha / \beta >0$, $d=\delta /2 >0$, we finally get the following canonical equations:
\begin{enumerate}
	\item[\bf (I)] \textit{Hyperbolic of first type quadric surfaces of revolution:} 
	\begin{enumerate}
		\item Lorentzian ellipsoid: $x_3^2-x_2^2=\alpha^2-\beta^2 \,  x_1^2$, i.e.\
		$$\frac{x_1^2}{b^2}-\frac{x_2^2}{a^2}+\frac{x_3^2}{a^2}=1, $$
		\item Lorentzian one-sheet hyperboloid: $x_3^2-x_2^2=\alpha^2+\beta^2 \, x_1^2$, i.e.\
		$$-\frac{x_1^2}{b^2}-\frac{x_2^2}{a^2}+\frac{x_3^2}{a^2}=1, $$
		\item Lorentzian two-sheets hyperboloid: $x_3^2-x_2^2=-\alpha^2+\beta^2\,  x_1^2$, i.e.\
		$$\frac{x_1^2}{b^2}+\frac{x_2^2}{a^2}-\frac{x_3^2}{a^2}=1, $$
		\item Lorentzian paraboloid: $x_3^2-x_2^2=\delta \, x_1$, i.e.\
		$$-x_2^2+x_3^2=2d \, x_1. $$
	\end{enumerate}
	\item[\bf (II)] \textit{Hyperbolic of second type quadric surfaces of revolution:} 
	\begin{enumerate}
		\item Lorentzian ellipsoid: $x_2^2-x_3^2=\alpha^2-\beta^2 \,  x_1^2$, i.e.\
		$$\frac{x_1^2}{b^2}+\frac{x_2^2}{a^2}-\frac{x_3^2}{a^2}=1, $$
		\item Lorentzian one-sheet hyperboloid: $x_2^2-x_3^2=\alpha^2+\beta^2 \, x_1^2$, i.e.\
		$$-\frac{x_1^2}{b^2}+\frac{x_2^2}{a^2}-\frac{x_3^2}{a^2}=1, $$
		\item Lorentzian two-sheets hyperboloid: $x_2^2-x_3^2=-\alpha^2+\beta^2\,  x_1^2$, i.e.\
		$$\frac{x_1^2}{b^2}-\frac{x_2^2}{a^2}+\frac{x_3^2}{a^2}=1, $$ 
		\item Lorentzian paraboloid: $x_2^2-x_3^2=\delta \, x_1$, i.e.\
		$$ x_2^2-x_3^2=2d \, x_1 .$$
	\end{enumerate}
	\item[\bf (III)] \textit{Elliptic quadric surfaces of revolution:} 
	\begin{enumerate}
		\item Lorentzian ellipsoid: $x_1^2+x_2^2=\alpha^2-\beta^2 \,  x_3^2$, i.e.\
		$$\frac{x_1^2}{a^2}+\frac{x_2^2}{a^2}+\frac{x_3^2}{b^2}=1, $$
		\item Lorentzian one-sheet hyperboloid: $x_1^2+x_2^2=\alpha^2+\beta^2 \, x_3^2$, i.e.\
	 $$\frac{x_1^2}{a^2}+\frac{x_2^2}{a^2}-\frac{x_3^2}{b^2}=1, $$
		\item Lorentzian two-sheets hyperboloid: $x_1^2+x_2^2=-\alpha^2+\beta^2\,  x_3^2$, i.e.\
		 $$-\frac{x_1^2}{a^2}-\frac{x_2^2}{a^2}+\frac{x_3^2}{b^2}=1, $$
		\item Lorentzian paraboloid: $x_1^2+x_2^2=\delta \, x_3$, i.e.\
		$$x_1^2+x_2^2=2 d \, x_3.$$
	\end{enumerate}
	\item[\bf (IV)] \textit{Parabolic quadric surfaces of revolution:} 
	\begin{enumerate}
		\item  $x_1^2+x_2^2-x_3^2=\alpha^2-\beta^2 \,  (x_1-x_3)^2$, i.e.\
		$$ \frac{x_1^2+x_2^2-x_3^2}{\alpha^2}+\frac{\beta^2(x_1-x_3)^2}{\alpha^2}=1, $$
		\item  $x_1^2+x_2^2-x_3^2=\alpha^2+\beta^2 \, (x_1-x_3)^2$, i.e.\
		$$ \frac{x_1^2+x_2^2-x_3^2}{\alpha^2}-\frac{\beta^2(x_1-x_3)^2}{\alpha^2  }=1, $$
		\item  $x_1^2+x_2^2-x_3^2=-\alpha^2+\beta^2\,  (x_1-x_3)^2$, i.e.\
		$$ -\frac{x_1^2+x_2^2-x_3^2}{\alpha^2}+\frac{\beta^2(x_1-x_3)^2}{\alpha^2 }=1, $$
		\item  $x_1^2+x_2^2-x_3^2=-\alpha^2-\beta^2\,  (x_1-x_3)^2$, i.e.\
		$$ -\frac{x_1^2+x_2^2-x_3^2}{\alpha^2}-\frac{\beta^2(x_1-x_3)^2}{\alpha^2}=1, $$
	\end{enumerate}
	with $\alpha >0$, $\beta \geq  0$.
\end{enumerate}
The latter case (IV) allows for a different configuration of the radii functions. In fact, it makes no sense to consider case (IV)(d) as the expected case $x_1^2+x_2^2-x_3^2=\delta (x_1-x_3)$, because then \eqref{eq:Imp_u} leads to $u_\epsilon (v) = \delta >0$ and this contradicts \eqref{eq:K4}.

\begin{remark}\label{rm:coincidence}
With the above nomenclature, it is obvious that the canonical equations of cases (I)-(a) and (II)-(c) coincide.
The same happens to cases (I)-(c) and (II)-(a). But recall that the surfaces belonging to (II) are always timelike.
So the coincidence should be understood only with the timelike region of the hyperbolic of first type Lorentzian ellipsoid and Lorentzian two-sheets hyperboloid.
\end{remark}

\begin{remark}\label{rm:umb_quadric}
We point out that the hyperbolic plane $\H^2_+(R)$, $R>0$, can be seen as a  non-degenerate quadric surface of revolution; it corresponds to $a=b=R$ in cases (I)-(b) and (III)-(c), and to $\beta =0$, $\alpha=R$, in cases (IV)-(c)(d).
The same happens to the de Sitter 2-space $\s^2_1(R)$, $R>0$; it corresponds to $a=b=R$ in cases (I)-(c), (II)-(a) and (III)-(b), and to $\beta =0$, $\alpha=R$, in cases (IV)-(a)(b). 
\end{remark}

Using Remark \ref{Re:K^2}, it is a long straightforward exercise to arrive at the explicit expression of the geometric linear momentum of the generatrix curve (ellipse, hyperbola or parabola) of each non-degenerate quadric surface of revolution in $\L^3$ defined above. In addition, using Theorem \ref{Th:kmkp}, it is then easy to check that all of them satisfy a cubic Weingarten relation collected in the following list and we can also specify the spacelike and timelike regions of each non-degenerate quadric surface of revolution in $\L^3$:

\begin{enumerate}
	\item[\bf (I)] \textit{Hyperbolic of first type quadric surfaces of revolution:} 
	\begin{enumerate}
		\item Lorentzian ellipsoid: generatrix ellipse $x^2/b^2+z^2/a^2=1$,
		$$\mathcal K (z)^2=\frac{\epsilon \, b^2 z^2}{(a^2+b^2)z^2-a^4}, \quad
		 k_{\rm m}=-\epsilon \, \frac{a^4}{b^2}k_{\rm p}^3.$$
		 $$ {\rm Spacelike \ (resp.\ timelike)\ if \ } z^2 > \frac{a^4}{a^2+b^2} \,\left({\rm resp.\ if \  } z^2 < \frac{a^4}{a^2+b^2}\right).$$
		\item Lorentzian one-sheet hyperboloid: generatrix hyperbola $-x^2/b^2+z^2/a^2=1$,
		$$\mathcal K (z)^2=\frac{\epsilon \, b^2 z^2}{(b^2-a^2)z^2+a^4}, \quad
		k_{\rm m}=\epsilon \, \frac{a^4}{b^2}k_{\rm p}^3.$$
		 $$ {\rm Spacelike \ (resp.\ timelike)\ if \ } z^2 < \frac{a^4}{a^2-b^2} \,\left({\rm resp.\ if \  } z^2 > \frac{a^4}{a^2-b^2}\right), \, a^2>b^2.$$
		 When $a^2 \leq b^2$, it is always spacelike.
		\item Lorentzian two-sheets hyperboloid: generatrix hyperbola $x^2/b^2-z^2/a^2=1$,
	$$\mathcal K (z)^2=\frac{\epsilon \, b^2 z^2}{(b^2-a^2)z^2-a^4}, \quad
	k_{\rm m}=-\epsilon \, \frac{a^4}{b^2}k_{\rm p}^3.$$
	 $$ {\rm Spacelike \ (resp.\ timelike)\ if \ } z^2 > \frac{a^4}{b^2-a^2} \,\left({\rm resp.\ if \  } z^2 < \frac{a^4}{b^2-a^2}\right),  \, b^2>a^2.$$
	 When $b^2 \leq a^2$, it is always timelike.
		\item Lorentzian paraboloid: generatrix parabola $z^2=2d\,x$,
		$$\mathcal K (z)^2=\frac{\epsilon \,  z^2}{z^2-d^2}. \quad
		k_{\rm m}=-\epsilon \, d^2 k_{\rm p}^3,$$
		 $$ {\rm Spacelike \ (resp.\ timelike)\ if \ } |z| > d \,({\rm resp.\ if \  } |z|<d).$$
	\end{enumerate}
 Recall that $a >0$, $b >0$, $d >0$.
 See Figure \ref{fig:HyperIquadrics}.
 \begin{figure}[h]
 	\begin{center}
 		\includegraphics[width=14.5cm]{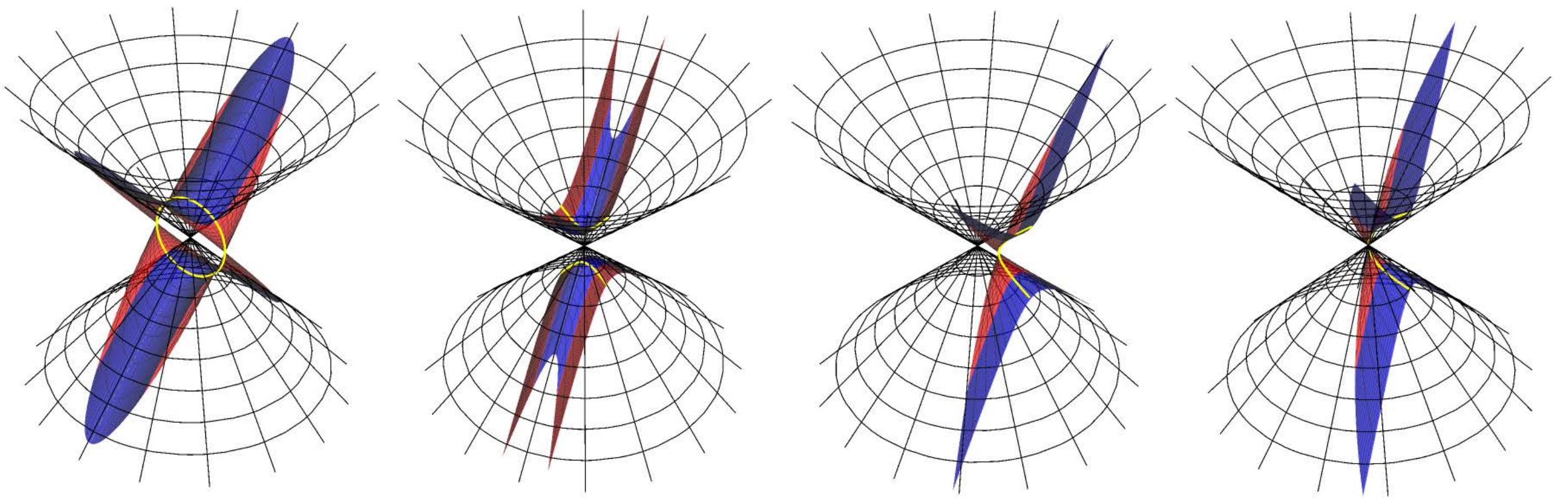}
 		\caption{Hyperbolic of first type quadric surfaces of revolution in $\L^3$ (in blue, the spacelike region; in red, the timelike region).}
 		\label{fig:HyperIquadrics}
 	\end{center}
 \end{figure}

	\item[\bf (II)] \textit{Hyperbolic of second type quadric surfaces of revolution:}
	\begin{enumerate}
		\item Lorentzian ellipsoid: generatrix ellipse $x^2/b^2+y^2/a^2=1$,
		$$\mathcal K (y)^2=\frac{b^2 y^2}{(b^2-a^2)y^2+a^4}, \quad
		k_{\rm m}= \frac{a^4}{b^2}k_{\rm p}^3.$$
		$$ {\rm If \ } a^2 > b^2, \ y^2 < \frac{a^4}{a^2-b^2};  {\rm \ if \ } a^2 \leq b^2, \ y \in (-\infty,+\infty) .$$
		\item Lorentzian one-sheet hyperboloid: generatrix hyperbola $-x^2/b^2+y^2/a^2=1$,
		$$\mathcal K (y)^2=\frac{b^2 y^2}{(a^2+b^2)y^2-a^4}, \, y^2 > \frac{a^4}{a^2+b^2} ,\quad
		k_{\rm m}= -\frac{a^4}{b^2}k_{\rm p}^3.$$ 
		\item Lorentzian two-sheets hyperboloid:  generatrix hyperbola $x^2/b^2-y^2/a^2=1$,
		$$\mathcal K (y)^2=\frac{b^2 y^2}{(a^2+b^2)y^2+a^4}, \, y \in (-\infty,+\infty) , \quad
		k_{\rm m}= \frac{a^4}{b^2}k_{\rm p}^3.$$ 
		\item Lorentzian paraboloid: generatrix parabola $y^2=2d\,x$,
		$$\mathcal K (y)^2=\frac{y^2}{y^2+d^2}, \, y \in (-\infty,+\infty) , \quad
		k_{\rm m}=d^2  k_{\rm p}^3.$$
	\end{enumerate}
 Recall that $a >0$, $b >0$, $d >0$.
See Figure \ref{fig:HyperIIquadrics}.
\begin{figure}[h]
	\begin{center}
		\includegraphics[width=14.5cm]{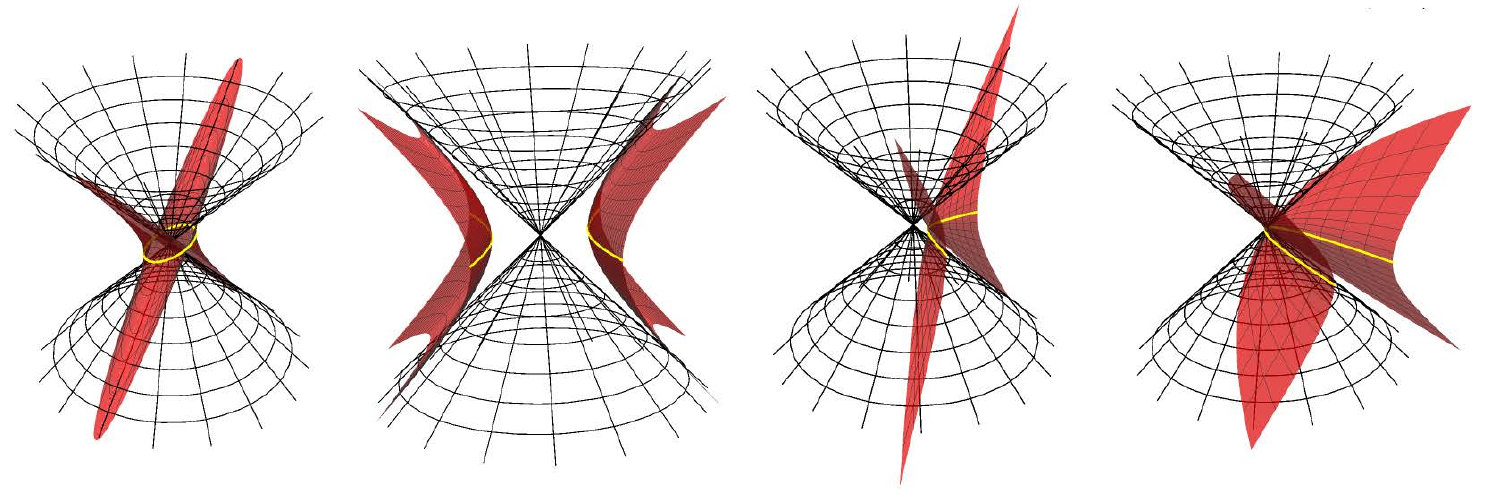}
		\caption{Hyperbolic of second type quadric surfaces of revolution in $\L^3$.}
		\label{fig:HyperIIquadrics}
	\end{center}
\end{figure}

	\item[\bf (III)] \textit{Elliptic quadric surfaces of revolution:} 
	\begin{enumerate}
		\item Lorentzian ellipsoid: generatrix ellipse $x^2/a^2+z^2/b^2=1$,
		$$\mathcal K (x)^2=\frac{\epsilon \, b^2 x^2}{a^4-(a^2+b^2)x^2}, \quad
		k_{\rm m}=\epsilon \, \frac{a^4}{b^2}k_{\rm p}^3.$$
		$$ {\rm Spacelike \ (resp.\ timelike)\ if \ } x^2 < \frac{a^4}{a^2+b^2} \,\left({\rm resp.\ if \  } x^2 > \frac{a^4}{a^2+b^2}\right).$$
		\item Lorentzian one-sheet hyperboloid: generatrix hyperbola $x^2/a^2-z^2/b^2=1$,
		$$\mathcal K (x)^2=\frac{\epsilon \, b^2 x^2}{(a^2-b^2)x^2-a^4}, \quad
		k_{\rm m}=-\epsilon \, \frac{a^4}{b^2}k_{\rm p}^3.$$
		$${\rm Spacelike \ (resp.\ timelike)\ if \ } x^2 > \frac{a^4}{a^2-b^2} \,\left({\rm resp.\ if \  } x^2 < \frac{a^4}{a^2-b^2}\right), \, a^2>b^2.$$
		When $a^2 \leq b^2$, it is always timelike.
		\item Lorentzian two-sheets hyperboloid: generatrix hyperbola $-x^2/a^2+z^2/b^2=1$,
		$$\mathcal K (x)^2=\frac{\epsilon \, b^2 x^2}{a^4+(a^2-b^2)x^2}, \quad
		k_{\rm m}=\epsilon \, \frac{a^4}{b^2}k_{\rm p}^3.$$
		$$ {\rm Spacelike \ (resp.\ timelike)\ if \ } x^2 < \frac{a^4}{b^2-a^2} \,\left({\rm resp.\ if \  } x^2 > \frac{a^4}{b^2-a^2}\right), \, a^2<b^2. $$
		When $a^2 \geq b^2$, it is always spacelike.
		\item Lorentzian paraboloid: generatrix parabola $x^2=2d\,z$,
		$$\mathcal K (x)^2=\frac{\epsilon \,  x^2}{d^2-x^2}, \quad
		k_{\rm m}=\epsilon \, d^2 k_{\rm p}^3.$$
		 $$ {\rm Spacelike \ (resp.\ timelike)\ if \ } |x| < d \,({\rm resp.\ if \  } |x|>d).$$
	\end{enumerate}
Recall $a >0$, $b >0$, $d >0$.
See Figure \ref{fig:Ellipticquadrics}.
\begin{figure}[h]
	\begin{center}
		\includegraphics[width=14.5cm]{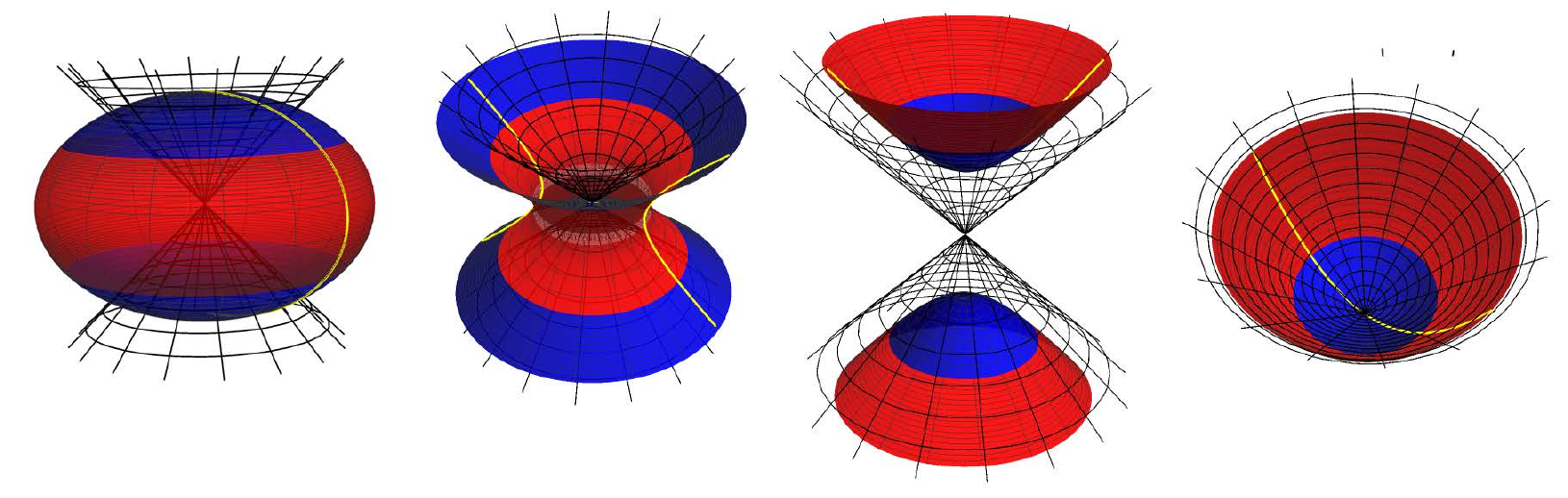}
		\caption{Elliptic quadric surfaces of revolution in $\L^3$ (in blue, the spacelike region; in red, the timelike region).}
		\label{fig:Ellipticquadrics}
	\end{center}
\end{figure}
   
	\item[\bf (IV)] \textit{Parabolic quadric surfaces of revolution:} 
	\begin{enumerate}
		\item  Generatrix hyperbola $u=\alpha^2/v-\beta^2 \, v$,
		
		 $\mathcal K (v)^2=\frac{-\epsilon \,  v^2}{\alpha^2 + \beta^2 v^2}$, therefore $\epsilon =-1$, that is, always timelike;
		$	k_{\rm m}= \alpha^2 k_{\rm p}^3$.
		\item  Generatrix hyperbola $u=\alpha^2/v+\beta^2 \, v$, 
		
		$\mathcal K (v)^2=\frac{\epsilon \,  v^2}{\beta^2 v^2 - \alpha^2}$;
		$	k_{\rm m}= - \epsilon \, \alpha^2 k_{\rm p}^3$.
		$$ {\rm Spacelike \ (resp.\ timelike)\ if \ } v^2 > \alpha^2/\beta^2 \,({\rm resp.\ if \  } v^2 < \alpha^2/\beta^2), \beta \neq 0.$$
		\item  Generatrix hyperbola $u=-\alpha^2/v+\beta^2 \, v$,
		
		$\mathcal K (v)^2=\frac{\epsilon \,  v^2}{\alpha^2 + \beta^2 v^2}$, therefore $\epsilon =1$, that is, always spacelike;
		$	k_{\rm m}=  \alpha^2 k_{\rm p}^3$.
		\item Generatrix hyperbola $u=-\alpha^2/v-\beta^2 \, v$,
		
		  $\mathcal K (v)^2=\frac{\epsilon \,  v^2}{\alpha^2-\beta^2 v^2}$;
		$	k_{\rm m}= \epsilon \, \alpha^2 k_{\rm p}^3$.
			$$ {\rm Spacelike \ (resp.\ timelike)\ if \ } v^2 < \alpha^2/\beta^2 \,({\rm resp.\ if \  } v^2 > \alpha^2/\beta^2), \beta \neq 0.$$
	\end{enumerate}
\end{enumerate}
Recall $\alpha >0$, $\beta >0$.
See Figure \ref{fig:Parabolicquadrics}.
\begin{figure}[h]
	\begin{center}
		\includegraphics[width=14.5cm]{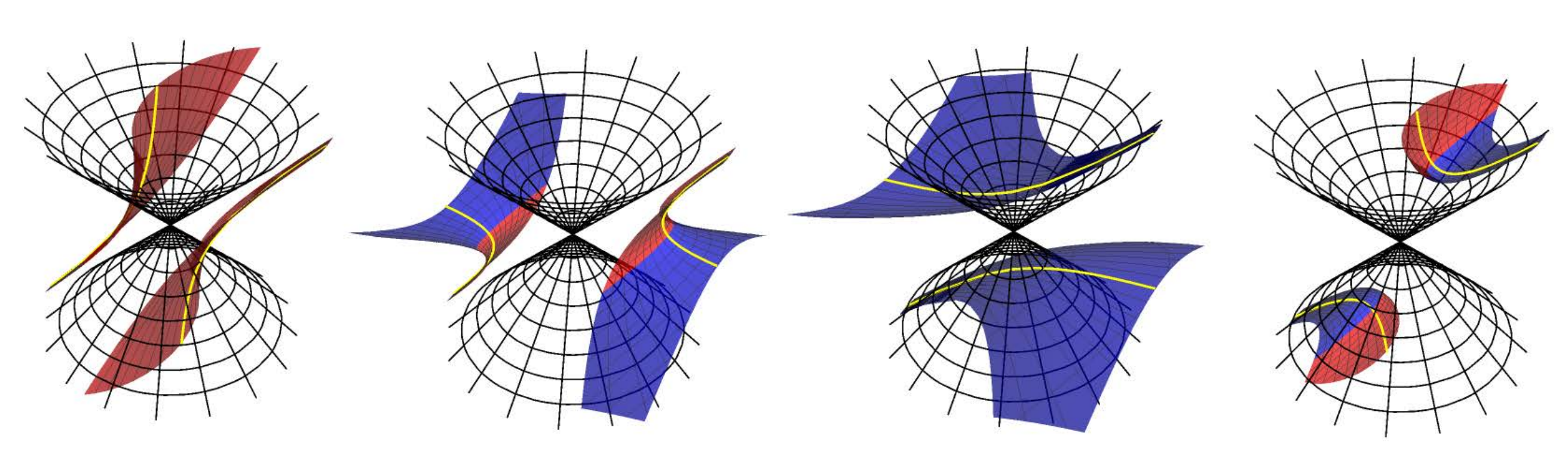}
		\caption{Parabolic quadric surfaces of revolution in $\L^3$ (in blue, the spacelike region; in red, the timelike region).}
		\label{fig:Parabolicquadrics}
	\end{center}
\end{figure}

\subsection{Classification}\label{SectClassification}

We are now in a position to establish our main result in this section characterizing all the quadric surfaces of revolution introduced in Section \ref{SectQuadric} in terms of a specific cubic Weingarten relation.

\begin{theorem}
	\label{Th:W-cubic}
	The only rotational surfaces in $\L^3$ satisfying $k_{\text m}=\mu\,  k_{\text p}^3$, $\mu \neq 0$, are the planes and the non-degenerate quadric surfaces of revolution described in Section \ref{SectQuadric}.
\end{theorem}

\begin{remark}\label{rm:mu0cones}
	We remark that if we take $\mu =0$ in the above cubic Weingarten relation, we also recover the Lorentzian rotational cones (see Example \ref{ex:Kcte}) and the right cylinders (see Remark \ref{Re:Cylinders}) of $\L^3$.
\end{remark}

\begin{proof}
	On the one hand, the spacelike and timelike planes obviously satisfy $k_{\text m}=\mu\,  k_{\text p}^3$, for any $\mu \neq 0$. In addition, we have checked in Section \ref{SectQuadric} that all the non-degenerate quadric surfaces of revolution satisfy $k_{\text m}=\mu\,  k_{\text p}^3$, for suitable $\mu \neq 0$ depending on the positive parameters $a$, $b$, $d$, $\alpha$ and $\epsilon =\pm 1$. 
	
	On the other hand, using Corollary \ref{cor:edoW},
	the~cubic Weingarten relation $k_{\text m}=\mu\,  k_{\text p}^3$ translates into the separable~o.d.e.
	$$\mathcal K ' (r)=\mu\, \mathcal K(r)^3/r^3. $$
	
	Its constant solution $\mathcal K \equiv 0$ leads to the spacelike and timelike planes (see Example \ref{ex:K0}).
	Its non-constant solution is given by
	\begin{equation}\label{eq:K cubic}
		\mathcal{K}(r)=  \frac{r}{\sqrt{\mu + c\,r^2}}  , \, c\in\mathbb R, \, \mu + c\,r^2 > 0.
	\end{equation}
	
	We first observe that if $c=0$, we have that $\mathcal{K}(r)= r/ \sqrt \mu$, $\mu >0$, and taking into account Example \ref{ex:K=r/R}, we arrive at the hyperbolic plane or de Siiter 2-space of radius $R=\sqrt \mu$ (see Remark \ref{rm:umb_quadric}). 
	
	Using Theorem \ref{Th:Kdetermine}, we are going to identify the rotational surfaces $S_\gamma^l$ uniquely determined, up~to $l$-translations, by~the one parameter family of geometric linear momenta (depending on $ c$) given in \eqref{eq:K cubic}; see also  Corollary \ref{cor:Param}. 
	For that purpose, we analyze the different classes of rotational surfaces in $\L^3$ and we also distinguish cases according to the signs of $\epsilon  \mu$ and $\epsilon c$.
	
	\begin{enumerate}
		
		\item[\bf (I)] \textit{Hyperbolic of first type rotational surfaces:} Now \eqref{eq:K cubic} is
		\begin{equation}\label{eq:KI}
			\mathcal{K}(z)=  \frac{z}{\sqrt{\mu + c\,z^2}}  , \, c\in\mathbb R, \, \mu + c\,z^2 > 0.
		\end{equation}
	\begin{itemize}
		\item If $\epsilon \mu <0$, we separate in turn three~possibilities:
		\begin{itemize}
			\item[(i)] $\epsilon c>1$: Then $a^2=\frac{-\epsilon \mu}{\epsilon c-1}$ is well defined, and putting $b^2=\frac{a^4}{-\epsilon \mu}$ we conclude that \eqref{eq:KI} exactly gives the geometric linear momentum corresponding to case (I)-(a).
			\item[(ii)] $\epsilon c<1$: We now define $a^2=\frac{-\epsilon \mu}{1-\epsilon c}$ and $b^2=\frac{a^4}{-\epsilon \mu}$. Then we obtain that \eqref{eq:KI} leads to the geometric linear momentum corresponding to case (I)-(c).
			\item[(iii)] $\epsilon c=1$: We define $d^2=-\epsilon \mu$ and conclude that \eqref{eq:KI} provides the geometric linear momentum corresponding to case (I)-(d).
		\end{itemize}
	    \item If $\epsilon \mu >0$, using that $\mathcal K (z)^2 -\epsilon >0$ (see Corollary \ref{cor:Param}(1)), we easily deduce $\epsilon c<1$. Then  $a^2=\frac{\epsilon \mu}{1-\epsilon c}$ is well defined and putting $b^2=\frac{a^4}{\epsilon \mu}$, we finish that \eqref{eq:KI} exactly gives the geometric linear momentum corresponding to case (I)-(b).
	\end{itemize}
	
	\item[\bf (II)] \textit{Hyperbolic of second type rotational surfaces:} Now \eqref{eq:K cubic} is
	 	\begin{equation}\label{eq:KII}
	\mathcal{K}(y)=  \frac{y}{\sqrt{\mu + c\,y^2}}  , \, c\in\mathbb R, \, \mu + c\,y^2 > 0.
\end{equation}
	\begin{itemize}
		\item If $\mu >0$, we discuss three~possibilities:
		\begin{itemize}
			\item[(i)] $ c<1$:  Then $a^2=\frac{\mu}{1-c}$ is well defined, and putting $b^2=\frac{a^4}{\mu}$, we conclude that \eqref{eq:KII} exactly gives the geometric linear momentum corresponding to case (II)-(a).
			\item[(ii)] $ c>1$: We now define $a^2=\frac{\mu}{c-1}$ and $b^2=\frac{a^4}{\mu}$. Then we obtain that \eqref{eq:KII} leads to the geometric linear momentum corresponding to case (II)-(c).
			\item[(iii)] $ c=1$: We define $d^2= \mu$ and conclude that \eqref{eq:KII} provides the geometric linear momentum corresponding to case (II)-(d).
		\end{itemize}
		\item If $\mu <0$, using that $1-\mathcal K (y)^2  >0$ (see Corollary \ref{cor:Param}(2)), we easily deduce $c>1$. Then $a^2=\frac{-\mu}{c-1}$ is well defined and putting $b^2=\frac{a^4}{-\mu}$, we finish that \eqref{eq:KII} exactly gives the geometric linear momentum corresponding to case (II)-(b).
	\end{itemize}

	\item[\bf (III)] \textit{Elliptic rotational surfaces:}   Now \eqref{eq:K cubic} is
	\begin{equation}\label{eq:KIII}
	\mathcal{K}(x)=  \frac{x}{\sqrt{\mu + c\,x^2}}  , \, c\in\mathbb R, \, \mu + c\,x^2 > 0.
\end{equation}
	\begin{itemize}
		\item If $\epsilon \mu >0$, we separate in turn three~possibilities:
		\begin{itemize}
			\item[(i)] $\epsilon c<-1$: Then $a^2=\frac{\epsilon \mu}{-(1+\epsilon c)}$ is well defined, and putting $b^2=\frac{a^4}{\epsilon \mu}$,  we conclude that \eqref{eq:KIII} exactly gives the geometric linear momentum corresponding to case (III)-(a).
			\item[(ii)] $\epsilon c>-1$: We now define $a^2=\frac{\epsilon \mu}{1+\epsilon c}$ and $b^2=\frac{a^4}{\epsilon \mu}$. Then we obtain that \eqref{eq:KIII} leads to the geometric linear momentum corresponding to case (III)-(c).
			\item[(iii)] $\epsilon c=-1$: We define $d^2=\epsilon \mu$ and conclude that \eqref{eq:KIII} provides the geometric linear momentum corresponding to case (III)-(d).
		\end{itemize}
		\item If $\epsilon \mu <0$, using that $\mathcal K (x)^2 +\epsilon >0$ (see Corollary \ref{cor:Param}(3)), we easily deduce $\epsilon c>-1$. Then $a^2=\frac{-\epsilon \mu}{1+\epsilon c}$ is well defined and putting $b^2=\frac{a^4}{-\epsilon \mu}$ we finish that \eqref{eq:KIII} exactly gives the geometric linear momentum corresponding to case (III)-(b).
	\end{itemize}

	\item[\bf (IV)] \textit{Parabolic rotational surfaces:}  Now \eqref{eq:K cubic} is
		\begin{equation}\label{eq:KIV}
	\mathcal{K}(v)=  \frac{v}{\sqrt{\mu + c\,v^2}}  , \, c\in\mathbb R, \, \mu + c\,v^2 > 0.
\end{equation}
We distinguish two cases and two subcases at each case:
	\begin{itemize}
	\item $\epsilon \mu <0$: 
	\begin{itemize}
		\item[(i)] $\epsilon c <0$: Then $\alpha^2=-\epsilon \mu$ and  $\beta^2 = -\epsilon c$ are well defined and
		we conclude that \eqref{eq:KIV} exactly gives the geometric linear momentum corresponding to case (IV)-(a).		
Recall that necessarily $\epsilon =-1$ in this case.
		\item[(ii)] $\epsilon c>0$: Now $\alpha^2=-\epsilon \mu$ and $\beta^2 = \epsilon c$ are well defined and
		we deduce that \eqref{eq:KIV} exactly gives the geometric linear momentum corresponding to case (IV)-(b).
	\end{itemize}
	\item $\epsilon \mu >0$:
	\begin{itemize}
		\item[(i)] $\epsilon c >0$: Then $\alpha^2=\epsilon \mu$ and  $\beta^2 = \epsilon c$ are well defined and
		we get that \eqref{eq:KIV} exactly gives the geometric linear momentum corresponding to case (IV)-(c).
		Recall that necessarily $\epsilon =1$ in this case.
		\item[(ii)] $\epsilon c<0$: Now $\alpha^2=\epsilon \mu$ and  $\beta^2 = -\epsilon c$ are well defined and
		we finish that \eqref{eq:KIV} exactly gives the geometric linear momentum corresponding to case (IV)-(d).
	\end{itemize}
   \end{itemize}

\end{enumerate}
This finishes the proof.
\end{proof}

\medskip

\noindent\textbf{Acknowledgments.} 
I.\ Castro is partially supported by the State Research Agency (AEI) via the grant no.\ PID2022-142559NB-I00, and the “Maria de Maeztu” Unit of Excellence IMAG, reference CEX2020001105-M, funded by MICIU/AEI/10.13039/501100011033 and ERDF/EU, Spain. I.\ Castro-Infantes is partially supported by the grant PID2021-124157NB-I00 funded by MCIN/AEI/ 10.13039/501100011033/ ‘ERDF A way of making Europe’, Spain; and by Comunidad Aut\'{o}noma de la Regi\'{o}n de Murcia, Spain, within the framework of the Regional Programme in Promotion of the Scientific and Technical Research (Action Plan 2022), by Fundaci\'{o}n S\'{e}neca, Regional Agency of Science and Technology, REF, 21899/PI/22.




\begin{thebibliography}{1}\bibliographystyle{alpha}


%
%
%


\bibitem{AG03} J.A.~Aledo and J.A.~Gálvez. {\em A Weierstrass representation for linear Weingarten
	spacelike surfaces of maximal type in the Lorentz–Minkowski space.} 
J.\ Math.\ Anal.\ Appl.\ {\bf 283} (2003), 25--45.

\bibitem{BL11} Ö.~Boyacioglu, and R.~L\'opez.
{\em Spacelike surfaces in Minkowski space satisfying a lineal relation between their principal curvatures.} 
Diff.\ Geom.\ Dyn.\ Syst.\  \textbf{13} (2011), 120--129.

\bibitem{BLS11} Ö.~Boyacioglu, R.~L\'opez and D.~Saglan.
{\em Linear Weingarten surfaces foliated by circles in Minkowski space.} 
Taiwanese J.\ Math.\  \textbf{15}(5) (2011), 1897--1917.

\bibitem{CC22} P.~Carretero and I.~Castro. {\em A new approach to rotational Weingarten surfaces.} 
Mathematics {\bf 2022} 10(4), 578; https://doi.org/10.3390/math10040578

\bibitem{CC23} P.~Carretero and I.~Castro. {\em A Geometric Characterization of The Quadric Surfaces of Revolution.} 
Rom.\ J.\ Math.\ Comput.\ Sci.\ {\bf 15} (2023)), 68--74. 

\bibitem{CC24} P.~Carretero and I.~Castro. {\em Rotational surfaces with prescribed curvatures.} 
Diff.\ Geom.\ Appl.\ {\bf 101}, 2025, 102298; https://doi.org/10.1016/j.difgeo.2025.102298

\bibitem{CCC25} P.~Carretero, I.~Castro and I.~Castro-Infantes. {\em Quadratic Rotational Weingarten Surfaces.} 
Rom.\ J.\ Math.\ Comput.\ Sci.\ {\bf 13} (2025)), 76--82. 


\bibitem{CCI16} I.~Castro and I.~Castro-Infantes.
{\em Plane curves with curvature depending on distance to a line.}
Diff.\ Geom.\ Appl.\ {\bf 44} (2016), 77--97.

\bibitem{CCI18} I.~Castro, I.~Castro-Infantes and J.~Castro-Infantes.
{\em Curves in Lorentz-Minkowski plane: elasticae, catenaries and grim-reapers}.
Open Math.\ \textbf{16} (2018), 747--766.
%
\bibitem{CCI20} I.~Castro, I.~Castro-Infantes and J.~Castro-Infantes.
 {\em Curves in Lorentz-Minkowski plane with curvature depending on their position.}
Open Math.\ \textbf{18} (2020), 749--770.

\bibitem{CCIs20} I.~Castro, I.~Castro-Infantes and J.~Castro-Infantes.
{\em On a Problem of David Singer about Prescribing Curvature for Curves.}
Geom.\ Integrability \& Quantization {\bf 21} (2020), 100--117.

\bibitem{CCIs24} I.~Castro, I.~Castro-Infantes and J.~Castro-Infantes.
{\em Helicoidal minimal surfaces in the 3-sphere: an approach via spherical curves.}
Rev.\ Real Acad.\ Cienc.\ Exactas Fis.\ Nat.\ Ser.\ A-Mat.\  {\bf 118}, 77 (2024). https://doi.org/10.1007/s13398-024-01574-3

\bibitem{CL97} W.~Chen and H.~Li. {\em Spacelike Weingarten surfaces in $\R^3_1$ and the 
	sine-Gordon equation}.  
J.\ Math.\ Anal.\ Appl.\ {\bf 214} (1997), 459--474.

\bibitem{Ch45} S.S~Chern.
{\em Some new characterizations of the Euclidean sphere.}
Duke Math.\ J.\ {\bf 12} (1945), 279--290.

\bibitem{dS21} L.C.B.~da~Silva.
{\em Surfaces of revolution with prescribed mean and skew curvatures in Lorentz-Minkowski space.}
T\^{o}hoku Math.\ J.\ (2) {\bf 73}(3) (2021), 317--339. https://doi.org/10.2748/tmj.20190729
%
%
%
\bibitem{Du10} U.~Dursun.
{\em Rotation hypersurfaces in Lorentz-Minkowski space with constant mean curvature.}
 Taiwanese J.\ Math.\ {\bf 14} (2010), 685--705.
%


\bibitem{HN83} J.I.~Hano and K.~Nomizu.
{\em On isometric immersions of the hyperbolic plane into the Lorentz-Minkowski space and the Monge-Amp\`{e}re equation of a certain type.}
Math.\ Ann.\ \textbf{262} (1983), 245-253.

\bibitem{HN84} J.I.~Hano and K.~Nomizu.
{\em Surfaces of revolution with constant mean curvature in Lorentz-Minkowski space.}
T\^{o}hoku Math.\ J.\  \textbf{36} (1984), 427-437.

\bibitem{H51} H.~Hopf.
{\em \"{U}ber Fl\"{a}chen mit einer Relation zwischen den Hauptkr\"{ummungen}.}
Math.\ Nachr.\ {\bf 4} (1951), 232--249.


\bibitem{Ko83} O.~Kobayashi.
{\em Maximal surfaces in the 3-dimensional Minkowski space $\L^{3}$}.
Tokyo J.\ Math.\ {\bf 6} (1983), 297--309.


\bibitem{K15} W.~Kühnel.
{\em Differential Geometry, Curves-Surfaces-Manifolds.}
AMS, 2015.

\bibitem{KS05} W.~Kühnel and M.~Steller.
{\em On closed Weingarten surfaces.}
Monatsh.\ Math.\ {\bf 146} (2005), 113--126.

\bibitem{LV06} S.~Lee and J.H.~Varnado.
{\em Spacelike constant mean curvature surfaces of revolution in Minkowski 3-space.}
Diff.\ Geom.\ Dyn.\ Syst.\  \textbf{8} (2006), 144--165.

\bibitem{LV07} S.~Lee and J.H.~Varnado.
{\em Timelike surfaces of revolution with constant mean curvature in Minkowski 3-space.}
Diff.\ Geom.\ Dyn.\ Syst.\  \textbf{9} (2007), 82--102.

\bibitem{Lo14} R.~L\'{o}pez.
{\em Differential Geometry of Curves and Surfaces in Lorentz-Minkowski space}.
Int.\ Electron.\ J.\ Geom.\ {\bf 1} (2014), 44--107.


%

\bibitem{S99} D.~Singer.
{\em Curves whose curvature depends on distance from the origin.}
Amer.\ Math.\ Monthly {\bf 106} (1999), 835--841.

%
\bibitem{VW90} I.~Van de Woestijne. 
{\em Minimal surfaces in the 3-dimensional Minkowski space}. 
In: Geometry and Topology of Submanifolds II, M.\ Boyon et al.\ (eds), World Scientific, Teaneck, NJ, 1990, 344--369.
%
\bibitem{W61} J.~Weingarten. 
{\em Ueber eine Klasse auf einander abwickelbarer Fl\"{a}chen}. 
J.\ Reine Angew.\ Math.\, \textbf{59} (1861), 382--393.

\bibitem{YZ21} D.~Yang and X.Y.~Zhu. {\em The complete classification of a class of new linear Weingarten surfaces in Minkowski space.} 
J.\ Math.\ Anal.\ Appl.\ {\bf 496} (2021), 124799.

\end{thebibliography}
\end{document}